\documentclass[12pt]{amsart}

\RequirePackage[l2tabu, orthodox]{nag}
\usepackage[T1]{fontenc}
\usepackage[utf8]{inputenc}
\usepackage{lmodern}
\usepackage[english]{babel}
\usepackage{color} \usepackage[usenames,dvipsnames]{xcolor}
\usepackage{amsrefs,amsthm,amssymb,amsmath}
\usepackage[pdftex]{graphicx}
\usepackage{enumerate}
\usepackage{mathtools}
\usepackage{caption}
\usepackage{appendix}

\numberwithin{equation}{section}
\numberwithin{figure}{section}

\let\oldtocsection=\tocsection
 
\let\oldtocsubsection=\tocsubsection
 
\let\oldtocsubsubsection=\tocsubsubsection
 
\renewcommand{\tocsection}[2]{\hspace{0em}\oldtocsection{#1}{#2}}
\renewcommand{\tocsubsection}[2]{\hspace{2em}\oldtocsubsection{#1}{#2}}
\renewcommand{\tocsubsubsection}[2]{\hspace{4em}\oldtocsubsubsection{#1}{#2}}

\def\comment#1{}

\newcommand{\R}{\mathbb{R}}
\newcommand{\N}{\mathbb{N}}
\newcommand{\Z}{\mathbb{Z}}
\newcommand{\T}{\mathbb{S}^1}
\newcommand{\Diff}{\mathsf{Diff}}
\newcommand{\Homeo}{\mathsf{Homeo}}

\newcommand{\eps}{\varepsilon}
\newcommand{\PSL}{\mathsf{PSL}}
\newcommand{\SL}{\mathsf{SL}}
\newcommand{\St}{\mathsf{St}} 

\newcommand{\cI}{\mathcal{I}}
\newcommand{\cJ}{\mathcal{J}}
\newcommand{\cG}{\mathcal{G}}
\newcommand{\wO}{\widetilde{O}}
\newcommand{\wTh}{\widetilde{\Theta}}

\DeclareMathOperator{\rel}{\mathsf{rel}}
\DeclareMathOperator{\Isom}{\mathsf{Isom}}
\DeclareMathOperator{\rot}{\mathsf{rot}}

\newtheorem{thm}{Theorem}[section]
\newtheorem{thmA}{Theorem}
\newtheorem*{thm*}{Theorem}
\newtheorem{lem}[thm]{Lemma}
\newtheorem{prop}[thm]{Proposition}
\newtheorem{cor}[thm]{Corollary}
\newtheorem*{cor*}{Corollary}

\newtheorem*{claim}{Claim}
\newtheorem{claim1}{Claim}

\theoremstyle{definition}
\newtheorem{dfn}[thm]{Definition}
\newtheorem*{conv}{Convention}

\theoremstyle{remark}
\newtheorem{rem}[thm]{Remark}
\newtheorem{ex}[thm]{Example}

\usepackage[margin=2.44cm]{geometry}

\usepackage{indentfirst}
\usepackage{microtype}

\usepackage[colorlinks=true,linkcolor=blue,citecolor=magenta]{hyperref}

\title[Ping-pong partitions and locally discrete groups of $\Diff_+^\omega(\T)$, I]{Ping-pong partitions and locally discrete groups of real-analytic circle diffeomorphisms, I: Construction}
\date{}
\author[J. Alonso, S. Alvarez, D. Malicet, C. Meni\~no Cot\'on, M. Triestino]{
Juan Alonso \and S\'ebastien Alvarez    \and
  Dominique Malicet \and Carlos Meni\~{n}o \and Michele Triestino
}
\begin{document}

\begin{abstract}
Following the recent advances in the study of groups of circle diffeomorphisms, we describe an efficient way of classifying the topological dynamics of locally discrete, 
finitely generated, virtually free subgroups of the group $\Diff^\omega_+(\T)$ of orientation preserving real-analytic circle diffeomorphisms, which include all subgroups of $\Diff^\omega_+(\T)$ acting with an invariant Cantor set.
An important tool that we develop, of independent interest, is the extension of classical ping-pong lemma to actions of fundamental groups of graphs of groups.
Our main motivation is  an old conjecture by P.\ R. Dippolito [Ann.~Math.~(2) 107 (1978), 403--453] from foliation theory, which we solve in this restricted but significant setting: this and other consequences of the classification will be treated in more detail in a companion work (by a slightly different list of authors).

\smallskip
{\noindent\footnotesize \textbf{MSC\textup{2010}:} Primary 37C85, 20E08. Secondary 20E06, 57S05.}
\end{abstract}

\maketitle


\section{Introduction}

In this work we study virtually free groups of real-analytic circle diffeomorphisms. Recall that a group is \emph{virtually free} if it contains a free subgroup of finite index. All the actions on the circle that we consider preserve the orientation, and we denote by $\Diff_+^\omega(\T)$ the group of order-preserving, real-analytic circle diffeomorphisms.
We are more specifically interested in those virtually free groups that are \emph{($C^1$) locally discrete} in $\Diff_+^\omega(\T)$: this is a condition that is stronger than simple discreteness, as it is not only required that the identity be isolated in $G$, but we want it to be isolated in restriction to every interval where there is the interesting dynamics of the group. Namely, let $\Lambda\subset \T$ be a minimal invariant compact set for $G$ (this is either the whole circle, or a Cantor set, or a finite orbit); $G$ is locally discrete if for every interval $I$ intersecting $\Lambda$, the restriction of the identity to $I$ is isolated among the set of restrictions in the group
\[
G\restriction_I=\{g\restriction_I\mid g\in G\}\subset C^1(I;\T),
\]
with respect to the $C^1$ topology. The case of minimal invariant Cantor set (called also \emph{exceptional}) is of special interest in this paper. By work of Rebelo \cite{Rebelo}, every finitely generated subgroup $G\subset\Diff^\omega_+(\T)$ with a minimal invariant Cantor set is locally discrete. Moreover, by a theorem of Ghys \cite{euler} such a group must be virtually free.

Note that when the action of $G$ has finite orbits, the minimal invariant compact set may not be unique. In fact, by the so-called Hector's lemma \cite{euler,proches,Navas2006,football}, in this case the locally discrete group $G\subset \Diff_+^\omega(\T)$ is \emph{virtually cyclic}, therefore \textit{a fortiori} the definition does not depend on the choice of $\Lambda$. Reciprocally, by the Denjoy--Koksma inequality \cite{herman,NT}, locally discrete, virtually cyclic groups must have finite orbits.
This has an important consequence: locally discrete, virtually cyclic groups are \emph{topologically classified} by the periodic orbits and their symmetries, and more precisely by their number and their dynamical behaviour (attracting/repelling).

As Bonatti and Langevin nicely explain in \cite{BL}, a classification problem usually divides into three parts.
\begin{enumerate}
\item First, one has to solve the problem of \emph{coding}. That is, given an action, one has to find a recipe to describe the action by a finite amount of data, in such a way that if for two actions one has the same data, then the actions are (semi-)conjugate.
\item Second, one has to settle the problem of \emph{realization}. That is, one has to detect which arrays of data actually come from a coding of an action.
\item Third, there is the problem of \emph{recognition}. That is, a given action may be encoded by two different arrays of data and one wants to be able to determine (algorithmically) when this happens.
\end{enumerate}

In the example above of virtually cyclic groups, the three facets of the topological classification problem are solved. As another major example, the topological classification of discrete subgroups of $\PSL(2,\R)$ coincides (see Goldman \cite{Goldman}) with the classification of hyperbolic surfaces (possibly with boundary, cusps and conic points).

Our main result is the \emph{first step for the topological classification} of general locally discrete, virtually free groups in $\Diff_+^\omega(\T)$. As for the simple case of virtually cyclic groups, we shall describe the conjugacy class with simple combinatorial and dynamical data (like the ``periodic orbits'' above): these will be encoded in a partition, that we call \emph{ping-pong partition}, associated with actions on trees and good generating sets. These are the generalization of more classical ping-pong partitions to groups acting on trees, first developed by Fenchel and Nielsen \cite{Fenchel-Nielsen} and later used successfully by Maskit in his combinations theorems \cite{Maskit}. They can also be thought as relatives of classical Markov partitions in dynamical systems; we discuss the relation between these two notions in the next section.

As for free groups it is natural to define a ping-pong partition associated with a given generating system, a good analogue for virtually free groups are \emph{orientation-preserving, cocompact, proper actions on locally finite trees}. In this precise setting, \emph{proper} means that the action has finite stabilizers. In the following we say that a \emph{marking} for a finitely generated, virtually free group $G$ is a proper action $\alpha:G\to\Isom_+(X)$ of $G$ on a locally finite tree, together with a connected fundamental domain $T\subset X$ for the action $\alpha$.

\begin{conv}
In the following, by \emph{virtually free group} we will always mean a finitely generated, virtually free group which is not virtually cyclic (unless explicitly stated).
\end{conv}

\begin{thmA}\label{t:main}
Let $G\subset \Diff^\omega_+(\T)$ be a locally discrete, virtually free group of real-analytic circle diffeomorphisms. For any marking $(\alpha:G\to \Isom_+(X),T)$, there  exists a proper  ping-pong partition for the action of $G$ on $\T$ (in the sense of Definition~\ref{d:markov-partition}).
\end{thmA}

As there are several definitions appearing in the literature of semi-conjugacy of actions on the circle (see \cite{GhysBounded,BFH,RigFlex,CD}), let us fix the following:

\begin{dfn}
Let $\rho,\rho':G\to \Homeo_+(\T)$ be two representations. They are \emph{semi-conjugate} if the following holds: there exist
\begin{itemize}
\item a monotone non-decreasing map $h:\R\to \R$ commuting with the integer translations and
\item two corresponding central lifts $\widehat{\rho},\widehat \rho':\widehat{G}\to \Homeo_{\Z}(\R)$  to homeomorphisms of the real line commuting with integer translations,
\end{itemize}
such that
\[
h\,\widehat{\rho}(\widehat{g})=\widehat{\rho}'(\widehat{g})\,h,\quad\text{for any }\widehat{g}\in \widehat{G}.
\]
\end{dfn}

\begin{thmA}\label{t:conj}
Let $\rho,\rho':(G,\alpha,T)\to \Homeo_+(\T)$ be two  representations of a virtually free group with a marked action $\alpha$ on a tree. Suppose that the actions on $\T$ have equivalent proper ping-pong partitions (in the sense of Definition~\ref{d:equivalence_PPP}). Then the actions are semi-conjugate.
\end{thmA}

\begin{rem}
In the statement of Theorem~\ref{t:conj}, the representations are automatically  injective: as we will see with Theorem~\ref{p:faithful}, a proper ping-pong partition for a virtually free group forces the action to be faithful.
\end{rem}

To give a more concrete picture, in the classical case of discrete groups $\Gamma$ in $\PSL(2,\R)$, the ping-pong partition that we consider corresponds to some coding for the geodesic flow on the unit tangent bundle $\PSL(2,\R)/\Gamma$.
However, our construction is more dynamical than geometrical, based on the recent work of Deroin, Kleptsyn and Navas \cite{DKN2014}, where Theorem~\ref{t:main} is proved for \emph{free groups}. In their work, this is the key for establishing important ergodic properties of virtually free groups of real-analytic circle diffeomorphisms, solving old conjectures by Ghys, Sullivan and Hector in this particular setting. 
Several other consequences of the  ping-pong partitions are described in the companion paper~\cite{MarkovPartitions2}, by a slightly different list of authors. We have decided to put the main focus of this first part on group-theoretical aspects, while the second one will treat concepts of more dynamical nature.
Let us however describe the main application that motivated this first work, whose proof appears in \cite{MarkovPartitions2}: ping-pong partitions allow to solve an old conjecture by Dippolito \cite{Dippolito} in the significant case of foliations arising as suspensions of group actions, in real-analytic regularity (see \cite{MarkovPartitions2} for more details, context and history).

\setcounter{thmA}{3}

\begin{thmA}[Dippolito conjecture in class $C^\omega$]\label{t-Dippolito}
	Let $G\subset\Diff_+^\omega(\T)$ be a group of real-analytic circle diffeomorphisms with an invariant Cantor set. Then the action of $G$ is semi-conjugate to an action by piecewise-linear homeomorphisms. More precisely, every such $G$ is semi-conjugate to a subgroup of Thompson's group $T$.
\end{thmA}

The idea of the proof is that if a group acts on the circle with a proper ping-pong partition (which is the case after Theorem \ref{t:main}), one can perturb the action so that generators are PL homeomorphisms admitting the same (or equivalent) ping-pong partition. Theorem \ref{p:faithful} guarantees that the action remains faithful and Theorem \ref{t:conj} implies that the semi-conjugacy class does not change. Finally, one can take a representative in the equivalence class of the ping-pong partition with the property that endpoints of intervals are all dyadic rationals (that is, in $\Z[\tfrac12]/\Z\subset \T$) so that the PL construction gives a subgroup of Thompson's group $T$.
This is not difficult in the case of a free group, but more technical arguments are needed for arbitrary virtually free groups, and the main ingredient is represented by \cite[Proposition 3.7]{MarkovPartitions2}.

If for the case of free groups Theorem \ref{t-Dippolito} is a rather elementary consequence of the work of Deroin, Kleptsyn and Navas \cite{DKN2014} (see \cite[Lemma 2.6]{isolated}), solving it for general groups requires the whole generality of Theorems \ref{t:main} and \ref{t:conj}. The main difficulty is that we have to consider actions of groups with \emph{torsion elements}, which is often a delicate task. Here we solve this problem by replacing the use of normal forms in \cite{DKN2014} by actions on trees: a marking of a virtually free group $G$ defines a partition of the boundary of the group $\partial G$ (\S\ref{sc:arboreal_part}), on which the group $G$ plays ``ping-pong'' (\S\ref{sc:arboreal_ping-pong}), which has a good behaviour when passing to finite-index subgroups (\S\ref{sc:finite_index}), and induces a ping-pong partition for the action on the circle, which corresponds to the construction of Deroin--Kleptsyn--Navas in the case of free groups (Section \ref{s:topskel_gen}). This is essentially the outline of the proof of Theorem~\ref{t:main}. For Theorem~\ref{t:conj}, we have to face the extra difficulty that is to formalize a good notion of ping-pong for virtually free groups: this should be good enough so that when introducing ping-pong partitions for actions on the circle (Definition~\ref{d:markov-partition}), the notion ensures that finitely many combinatorial data are enough to recover the semi-conjugacy class of the action (\S\ref{sc:proof_thmB}). This problem requires a long detour through Bass--Serre theory (Section \ref{s:BS}), which culminates in a generalized ping-pong lemma for fundamental groups of graphs of groups (Theorem~\ref{p:faithful}), a key result which is also of independent interest.

\smallskip

Observe that in this work we do not address the question of determining what  ping-pong partitions actually come from an action of a given group.  In basic cases as a free product like $F_n$ or $\Z_2*\Z_3$ ($\cong \PSL(2,\Z)$), a systematic description of all  ping-pong partitions seems out of the reach \cite{MarkovPartitions2,matsumoto-psl, isolated}.
Related to this problem, we prove however the following very general result:
\setcounter{thmA}{2}
\begin{thmA}\label{t:RealVirtFree}
Let $G\subset \Homeo_+(\T)$ be a finitely generated, virtually free group (possibly virtually cyclic). Then $G$ is \emph{free-by-finite cyclic}: there exists a free subgroup $F\subset G$ of index $m\in \N$ such that $G$ fits into a short exact sequence
\[
1\to F\to G\to \Z_m\to 0.
\]
Conversely, any finite cyclic extension of a free group can be realized as a locally discrete group of real-analytic circle diffeomorphisms. 
\end{thmA}

Neither we address the problem of recognition: this should pass through an understanding of how an automorphism of a virtually free group (which essentially changes the marked action on the tree)
modifies a  ping-pong partition (can the combinatorics of the partition be different? This problem should be first studied for free groups, for which we expect a negative answer). Observe that there is an extensive literature about decidability properties for virtually free groups (see to this purpose \cite{McCool,Krstic} and the survey \cite{Vogtmann}).

Finally, we highlight that virtually free groups constitute an important class of locally discrete subgroups of $\Diff_+^\omega(\T)$: it is conjectured \cite{FKone,tokyo,football} that such subgroups can only be virtually free or $C^\omega$-conjugate to a finite central extension of a Fuchsian group (cocompact lattice in $\PSL(2,\R)$). Currently, this has been validated for groups acting with an invariant Cantor set (recall that in \cite{euler} Ghys proves that such groups are virtually free), and for locally discrete groups acting minimally with the so-called property $(\star)$.

\begin{rem}
	In this article we consider only groups of homeomorphisms which preserve the orientation. Any subgroup $G\subset \Homeo(\T)$ contains a subgroup $G_+$ of index at most 2 which preserves the orientation. It is however against the spirit of this article to simplify the problems passing to a finite-index subgroup, and we should also treat virtually free subgroups with orientation-reversing elements. As a disclaimer, when making this simplification we are only leaving aside dynamics of dihedral type, which is not hard to understand, but would lead to heavier notation.
\end{rem}

\section{Markov partitions versus  ping-pong partitions}

In \cite{Bowen-Series}, Bowen and Series described a natural way to encode the action of a Fuchsian group $\Gamma\subset \PSL(2,\R)$ on the circle, constructing \emph{Markov partitions}. Informally speaking, a Markov partition is a finite collection of non-overlapping intervals of the circle, together with a map $T$ piecewise defined over each of these intervals, coinciding with the generator of the Fuchsian group that performs an expansion of the interval. In the construction, the map $T$ is \emph{expanding} and \emph{orbit equivalent} to the action of $\Gamma$ on the circle. Using this construction, Bowen and Series succeeded in studying several ergodic properties of the action of~$\Gamma$ on $\T$.

In a similar spirit, in \cite{DKN2009}, with motivation from the ergodic theory of foliations,  Deroin, Kleptsyn, and Navas associate a Markov partition with any locally discrete, finitely generated group $\Gamma$ of $C^2$ circle diffeomorphisms verifying what they call property~$(\star)$, or $(\Lambda\star)$. 
(Their construction is reminiscent of \cite{CantwellConlon1,CantwellConlon2}.)
In this case, the map $T$ that they define is also expanding and orbit equivalent to $\Gamma$. The explicit construction was pursued in \cite{FK2012_C_eng}.
In \cite{DKN2014}, the same authors directly work with locally discrete, free groups of real-analytic diffeomorphisms, and they build an associated 
Markov partition, thus showing that such a group necessarily has property~$(\star)$, or $(\Lambda\star)$. In the latter situation, the Markov partition actually comes from a \emph{ping-pong partition}, in the sense that the generators of the free group play ping-pong with the intervals of the partition.
In both situations the Markov partitions that Deroin--Kleptsyn--Navas consider, are obtained using techniques of control of the affine distortion.

Compared to the construction in \cite{DKN2014} of  ping-pong partitions, when dealing with \emph{virtually free} groups, we have to overpass the problem of having many ways for writing a given element as a product in the generating system (in particular we have to deal with amalgamated subgroups). 
For this reason, we slightly modify the construction of Deroin, Kleptsyn, and Navas. Actually, our approach is largely inspired by that of \cite{BP}, where Matsumoto studies semi-conjugacy classes of representations of surface groups, introducing the notion of \emph{basic partition}. In particular, we borrow from Matsumoto, and actually back from Maskit \cite{Maskit}, the idea that the action on a Bass--Serre tree induces a partition of the circle. We shall recall the basic notions of Bass--Serre's theory in \S\ref{Bass--Serre}. As the reader will see, a serious difficulty not appearing in \cite{BP} is that virtually free groups which are not free contain torsion elements.

\section{Ping-pong and the DKN construction for free groups}\label{s:DKN}

Before getting started with virtually free groups, we recall the main construction in \cite{DKN2014} for free groups.
In this section, we denote by $G$ a rank-$n$ free group. We choose $S_0$ a system of free generators for $G$, and write $S=S_0\cup S_0^{-1}$. We denote by $\|\cdot\|$ the word norm defined by $S$.
Recall that any element $g$ in $G$ may be written in unique way
\[
g=g_\ell\cdots g_1,\quad g_i\in S,
\]
with the property that if $g_i=s$ then $g_{i+1}\neq s^{-1}$ (we write compositions from right to left). This is called the \emph{normal form} of $g$.
For any $s\in S$, we define the set
\[
W_s:=\left \{g\in G\mid g=g_\ell\cdots g_1\text{ in normal form, with }g_1=s\right \}.
\]
If $X$ denotes the Cayley graph of $G$ with respect to the generating set $S$ (which is a $2n$-regular tree), with the right-invariant distance, then the sets $W_s$ are exactly the $2n$ connected components of $X\setminus\{\mathrm{id}\}$, with $W_s$ being the connected component containing $s$. The choice of the right-invariant distance gives the isometric \emph{right} action of $G$ on $X$: $x\in X\mapsto xg\in X$.
It is easy to see that the generators $S$ play ping-pong with the sets $W_s$:
for any $s\in S$, we have $\left (X\setminus W_{s^{-1}}\right )s\subset W_{s}$.
Using the action on the circle, we can push this ping-pong partition of the Cayley graph of $G$ to a partition of the circle into open intervals with very nice dynamical properties.
Given $s\in S$, we define the subset
\begin{equation}\label{eq:Usets}
U_{s}:=\left\{x\in\T\,\middle\vert\,\exists\text{ neighborhood }I_x\ni x\text{ s.t.~}\lim_{n\to\infty}\sup_{g\notin W_s,\|g\|\ge n}|g(I_x)|=0\right\}.
\end{equation}
In \cite{DKN2014} it is proved the following:
\begin{thm}[Deroin, Kleptsyn, and Navas]\label{t:DKNfree}
Let $G\subset \Diff_+^\omega(\T)$ be a finitely generated, locally discrete, free group of real-analytic circle diffeomorphisms, with minimal invariant set $\Lambda$. Let $S_0$ be a system of free generators for~$G$ and write $S=S_0\cup S_0^{-1}$.
Consider the collection $\{U_s\}_{s\in S}$ defined in \eqref{eq:Usets}. We have:
\begin{enumerate}[1.]
\item \label{i:open_free} every subset $U_s$  is open;

\item \label{i:finite_free} every subset $U_s$ has finitely many connected components;

\item \label{i:intersection_on_minimal} any two different subsets  $U_s$ have empty intersection inside the minimal invariant set $\Lambda$;

\item\label{i:cover} the union of the subsets $U_s$ covers all but finitely many points of $\Lambda$;

\item \label{i:ping-pong_free} if $s\in S$, $t\neq s$ then  
 $s(U_t)\subset U_{s^{-1}}$.
\end{enumerate}
\end{thm}

For $s\in S$, and connected component $I=(x_-,x_+)\subset U_s$ such that a right (respectively, left) neighborhood of $x_-$ (respectively, $x_+$) is contained in a connected component $J_-$ (respectively, $J_+$) of the complement $\T\setminus \Lambda$, denote by $\hat I$ the interval resulting from removing $J_-$ and $J_+$ (whenever they are non-empty) from $I$. Then denote by $\hat U_s$ the subset of $U_s$, formed by the union of the reduced intervals $\hat I$, for $I$ connected component of $U_s$. By construction, and $G$-invariance of $\Lambda$, the family $\{\hat U_s\}_{s\in S}$ satisfies all the properties \ref{i:open_free}.~through \ref{i:ping-pong_free}.~of Theorem~\ref{t:DKNfree}, and moreover they are pairwise disjoint, strengthening property \ref{i:intersection_on_minimal}. The properties enjoyed by the family $\{\hat U_s\}_{s\in S}$ correspond to what we want to be a ping-pong partition for a free group of circle homeomorphisms.

\begin{dfn}\label{d:ping-pong_free}
Let $G\subset \Homeo_+(\T)$ be a finitely generated, free group of circle homeomorphisms and let $S=S_0\cup S_0^{-1}$ be a symmetric free generating set. A collection $\{\hat U_s\}_{s\in S}$ of subsets of $\T$ is a \emph{ping-pong partition} for $(G,S)$ if it verifies the following conditions:
\begin{enumerate}[1.]
	\item every subset $\hat U_s$ is non-empty, union of finitely many open intervals;
	
	\item any two different subsets $\hat U_s$ have empty intersection;
	
	\item if $s\in S$, $t\neq s$ then  
	$s(\hat U_t)\subset \hat U_{s^{-1}}$.
\end{enumerate}

The \emph{skeleton} of the ping-pong partition is the data consisting of
\begin{enumerate}
\item the cyclic order in $\T$ of  connected components of $\bigcup_{s\in S}\hat U_s$, and 
\item for each $s\in S$, the assignment of connected components
\[
\lambda_s:\pi_0\left (\bigcup_{t\in S\setminus \{s\}}\hat  U_t\right )\to \pi_0\left (\hat U_{s^{-1}}\right )
\]
induced by the action.
\end{enumerate}
\end{dfn}

\begin{rem}
In \cite{DKN2014} the definition of the sets $U_s$ (there called $\widetilde{\mathcal{M}_\gamma}$) is slightly different based on a control on the sum of derivatives along geodesics in the group. Here the definition that we adopt is simply topological, as we consider how neighborhoods are contracted along geodesics in the group.
This difference in the definition leads to possibly different sets: one can show that $U_s$ contains the corresponding $\widetilde{\mathcal{M}_s}$, and the complement $U_s\setminus \widetilde{\mathcal{M}_s}$ is a finite number of points, which are topologically hyperbolic fixed points $x$ for some element $g$ in the group, but with derivative $g'(x)=1$. See \cite[Appendix A]{MarkovPartitions2} for more details.
\end{rem}

Even with the different definition, the proof of Theorem~\ref{t:DKNfree} proceeds as in \cite{DKN2014}. The hardest part is to prove property \ref{i:finite_free}, which is Lemma 3.30 in \cite{DKN2014}, and property \ref{i:cover}, which actually occupies most part of \cite{DKN2014} (because it comes as a consequence of property $(\star)$). For the other properties, we will give  more details and references when proving the more general Theorem~\ref{t:DKNmarkov3}.

\begin{rem}
Ping-pong partitions and skeletons for free groups already appear in \cite{MR,isolated} as ping-pong actions and configurations, respectively, with slightly different conventions. 
\end{rem}

Let us end this section by stating Theorem~\ref{t:conj} for the particular case of free groups. For this, we first need the following:

\begin{dfn}
Let $\rho_\nu:(G,S)\to \Homeo_+(\T)$, $\nu\in \{1,2\}$, be two injective representations of a finitely generated, free group with a marked symmetric free generating set $S=S_0\cup S_0^{-1}$. Let $\{U_s^{\nu}\}_{s\in S}$, be a ping-pong partition for $\rho_\nu(G,S)$, for $\nu\in \{1,2\}$.
We say that the two partitions are \emph{equivalent} if they have the same skeleton.
\end{dfn}

\begin{prop}
Let $\rho_\nu:(G,S)\to \Homeo_+(\T)$, $\nu\in \{1,2\}$, be two injective representations of a finitely generated, free group with a marked symmetric free generating set $S=S_0\cup S_0^{-1}$. Suppose that the actions on $\T$ have equivalent ping-pong partitions. Then the actions are semi-conjugate.
\end{prop}

The proof of this result is somehow classical (see \cite[Theorem 4.7]{BP} and \cite[Lemma 4.2]{MR}). We will extend this result to virtually free groups in Section \ref{s:topolconj}.

\section{Algebraic set-up: virtually free groups and Bass--Serre theory}\label{s:BS}

\subsection{Bass--Serre theory in a nutshell}\label{Bass--Serre}

The basic principle of Bass--Serre theory is to describe the algebraic structure of a group by studying actions on trees. By \emph{trees} we mean \emph{simplicial} trees, namely connected graphs without cycles or, more formally, connected, simply connected, one-dimensional CW complexes. The trees have \emph{oriented edges}, and the group actions on trees preserve the orientation.
A tree is \emph{locally finite} if each vertex has only finitely many edges attached to it.
Serre's book \cite{serre} is probably the exemplary reference for an introduction to the theory (see also \cite{DD} for a more detailed reference).
That such a point of view is fruitful, it can be already perceived in one of its first applications: a group~$G$ is free if and only if there exists a free action of $G$ on a tree (a free action is an action with trivial stabilizers). It then appears almost tautological that a subgroup of a free group is free (Nielsen--Schreier theorem). 
Also virtually free groups can be characterized in a similar way \cite[Corollary IV.1.9]{DD}:
\begin{thm}[Karrass, Pietrowski, and Solitar \cite{virtually_free}]\label{t:virt_free}
A group $G$ is virtually free if and only if there exists a proper action of $G$ on a tree. Moreover, if $G$ is finitely generated, the tree may be taken to be locally finite and the action cocompact.
\end{thm}

The fundamental notion in Bass--Serre theory is that of \emph{graph of groups} and its \emph{fundamental group}. The definitions are more intuitive when starting from a (left) action of a group $G$ on a tree $X$. Denote by $\overline{X}=(V,E)$ the quotient graph $X/G$, whose vertices correspond to the orbits of the vertices of $X$ under the action of $G$, and similarly do its edges. One defines the graph of groups $(\overline X; G_v, A_{e})$ attaching to any vertex $v\in V$ of $\overline{X}$ the stabilizer $G_v$ of a vertex in the orbit it represents, and to any edge $e\in E$ of $\overline{X}$, the stabilizer $A_{e}$ of one edge in the orbit (observe that $A_e=A_{\overline{e}}$, where $\overline{e}$ denotes the edge $e$ with reversed orientation). Choosing the representatives of the vertices in an appropriate way, one can see an edge group $A_{e}$ as the intersection of the vertex groups $G_{o(e)}$ and $G_{t(e)}$ (here we denote by $o(e)\in V$ and $t(e)\in V$ respectively the \emph{origin} and the \emph{target} of the oriented edge $e\in E$; if an edge is fixed, the two vertices on it are fixed, too). One formally writes this introducing \emph{boundary injections} $\alpha_e:A_e\to G_{o(e)}$, $\omega_e:A_e\to G_{t(e)}$.

Reciprocally, starting from a graph of groups $(\overline X; G_v, A_{e})$, one constructs a tree $X$, called the \emph{Bass--Serre tree}, and a group $G$, called the \emph{fundamental group} of $(\overline X; G_v, A_e)$ and denoted by $\pi_1(\overline X; G_v, A_e)$, such that:
\begin{itemize}
\item the group $G$ acts on $X$ and the quotient graph $X/G$ is isomorphic to $\overline{X}$;
\item the stabilizer of any vertex in $X$ is isomorphic to the group $G_v$, where $v\in V$ is the projection of the vertex to $\overline{X}$;
\item the stabilizer of any edge in $X$ is isomorphic to the group $A_{e}$, where $e\in E$ is the projection of the edge to $\overline{X}$.
\end{itemize}
Clearly, starting from a left action of a group $G$ on a tree $X$, then $G$ is isomorphic to the fundamental group of the graph of groups it defines, and the action on the Bass--Serre tree is conjugate to the action of $G$ on $X$. 
At the level of combinatorial group theory, from an action of a group $G$ on a tree $X$, one derives a \emph{presentation} for $G$. For this, we fix a \emph{spanning tree} $T=(V,E_T)\subset \overline{X}$. One has the isomorphism
\begin{equation}\label{eq:pres_fundamental_group}
G\cong \sbox0{$G_v,E\, \left\vert \begin{array}{lr}
\rel (G_v)& \text{for }v\in V,\\
\overline{e}=e^{-1}&\text{for }e\in E,\\
e=\mathrm{id} & \text{for }e\in E_T,\\
e^{-1}\alpha_e(g)e=\omega_e(g)&\text{ for }e\in E, g\in A_e\end{array}\right.$}
\mathopen{\resizebox{1.3\width}{\ht0}{$\Bigg\langle$}}
\usebox{0}
\mathclose{\resizebox{1.3\width}{\ht0}{$\Bigg\rangle$}}
\end{equation}
where $\rel (H)$ denotes the relations in the group $H$. The second and third sets of relations imply that $G$ is generated by the vertex groups $G_v$ and the symmetric set of edges $S\subset E\setminus E_T$ (here \emph{symmetric} means that one has the relations $\overline s=s^{-1}$ for $s\in S$). In particular, the classical fundamental group $\pi_1(\overline X,*)$ (which is a free group of rank $1-\chi(\overline X)$) is naturally a subgroup of $\pi_1(\overline X; G_v, A_e)$, freely generated by the symmetric set of edges $S$.

There are two standard particular classes of graphs of groups.
\begin{enumerate}
\item \emph{Amalgamated products.} When the graph of groups has simply two vertices and one edge, the fundamental group of the graph of groups is the amalgamated product of the vertex groups over the edge group:
\[G_1*_A G_2=\left \langle G_1,G_2\mid \rel(G_i), \alpha(a)=\omega(a)\text{ for }a\in A\right\rangle.\]
\item \emph{HNN extensions.} When the graph of groups has one only vertex and a self-edge, the fundamental group of the graph of groups is an HNN extension of the vertex group:
\[
G_0*_A=\left \langle G_0,s\mid \rel(G_0), s^{-1}\alpha(a)s=\omega(a)\text{ for }a\in A\right \rangle.
\]
\end{enumerate}

\begin{rem}\label{r:GraphsOfGroups}
Every fundamental group of a graph of groups is obtained by iterating amalgamated products and HNN extensions.
\end{rem}

\subsection{A first result: structure of virtually free groups acting on the circle}
\label{sc:virt_free_circle}

After this digression we come back to groups acting on the circle. From Karrass--Pietrowski--Solitar Theorem (Theorem \ref{t:virt_free}) we derive a very precise description of virtually free groups in $\mathsf{Homeo}_+(\T)$.
In \cite{euler}, Ghys proves that a finitely generated group $G$ of real-analytic circle diffeomorphisms with minimal invariant Cantor set has a finite-index, normal free subgroup $H$ such that the quotient $G/H$ is abelian (see Remarque 4.4 in \cite{euler}). Here we generalize this result:

\begin{thm}\label{t:StructureVirtFree}
Let $G\subset \mathsf{Homeo}_+(\T)$ be a finitely generated virtually free group (possibly virtually cyclic). Then $G$ is free-by-finite cyclic: there exists a normal free subgroup $H\subset G$ of finite index such that the quotient $G/H$ is cyclic.
\end{thm}

The proof takes inspiration from a standard procedure for finding a finite-index free subgroup inside a virtually free group \cite[\S II.3.6]{Dicks} (see also \cite[\S I.7-9]{DD}).

\begin{proof}
	By Theorem~\ref{t:virt_free}, the group $G$ is isomorphic to the fundamental group of a graph of groups $(\overline X; G_v,A_e)$ with finite vertex groups. Take 
\[m:=\mathsf{l.c.m.}_{v\in V}|G_v|,\]
and define a homomorphism $\pi:G\to \Z_m$ as follows. Observe that as $G_v\subset \Homeo_+(\T)$ is a finite subgroup, it must be cyclic.
\begin{claim}
	There exist homomorphic embeddings $\pi_v:G_v\hookrightarrow \Z_m$ such that $\pi_{o(e)}\restriction_{\alpha_e(A_{e})}$ and $\pi_{t(e)}\restriction_{\omega_e(A_{e})}$ coincide for any edge $e\in E$.
\end{claim}
\begin{proof}[Proof of Claim]
	Here we consider $\Z_m=(\tfrac1m\Z)/\Z=\{0,\tfrac1m,\ldots,\tfrac{m-1}{m}\}$. For any $v$, the rotation number $\rot:G_v\to \R/\Z$ has image $\{0,\frac{1}{|G_v|},\ldots,\frac{|G_v|-1}{|G_v|}\}$.
	As $|G_v|$ divides $m$, the map $\pi_v(g)=\rot(g)$ defined on $G_v$ takes values in $\Z_m$. The relations $e^{-1}\alpha_e(g)e=\omega_e(g)$ in the presentation \eqref{eq:pres_fundamental_group} say that the restrictions $\pi_{o(e)}\restriction_{\alpha_e(A_{e})}$ and $\pi_{t(e)}\restriction_{\omega_e(A_{e})}$ coincide, as the rotation number is invariant under conjugacy.
\end{proof}  
From the presentation \eqref{eq:pres_fundamental_group} of the fundamental group of a graph of groups, we get that the homomorphisms $\pi_v$ from the Claim extend to a morphism $\pi:G\to \Z_m$, defined on generators as follows:
\[\left \{\begin{array}{lr}
\pi(g)=\pi_v(g) & \text{if }g\in G_v,\\
\pi(e)=0 & \text{if }e\in E.
\end{array}\right.
\]
The kernel $H$ of this morphism has trivial intersection with every vertex group, so the kernel acts freely on the Bass--Serre tree of $G$ and hence is free.
\end{proof}

\begin{rem}
	The proof may suggest that the morphism $\pi:G\to\Z_m$ is given by the rotation number. This is not true, because the free subgroup $H$ is the kernel of $\pi$ and for a general group $G\subset \Homeo_+(\T)$ the rotation numbers of generators of $H$ can be arbitrary (think of the case $G=H$). However, one may want to interpret $\pi$ as a rotation number \emph{relative to $H$}.
\end{rem}

As in \cite[Proposition 3.22]{DW} (inspired by \cite[Proposition 2.4]{DG}), we have the following:

\begin{prop}\label{cor:perm}
Let $G$ be a free-by-finite cyclic group
\[
1\to H\to G\to \Z_m\to 0,
\]
then $G$ is the subgroup of a semi-direct product $F(S_0)\rtimes \Z_m$, where $F(S_0)$ denotes the free group over a finite set $S_0$ and $\Z_m$ acts by permutations of $S_0$. 
\end{prop}

\begin{rem}\label{rem:perm}
Actually  Proposition~\ref{cor:perm} is obtained by a slightly stronger statement: if one denotes by $\Gamma$ the quotient graph of the action of the finite-index, normal free subgroup $H\subset G$ on the Bass--Serre tree of $G$, then the quotient group $G/H\cong \Z_m$ acts by graph automorphisms of $\Gamma$. Let $S$ be the set of edges of the quotient graph $\Gamma$, and consider the free group
\[F:=\langle S \mid \overline s=s^{-1}\quad\text{for every }s\in S\rangle,\]
where $\overline s$ denotes the edge $s\in S$ with reversed orientation. 
As the action of $G$ on the Bass--Serre tree is without edge-inversion, there exists an orientation $S_0$ of the edges (that is, $S_0$ defines a partition $S_0\sqcup\overline S_0= S$) which is preserved by the action of the quotient $\Z_m=G/H$.
\end{rem}

\begin{ex}
In the case of $G=\PSL(2,\Z)\cong \Z_2*\Z_3$ one gets that $G$ is a subgroup of $F_6\rtimes \Z_6$. Denoting by $s_1,\ldots,s_6$ the generators of $F_6$, a generator of $\Z_6$ acts as a cyclic permutation of $\{s_1,\ldots,s_6\}$. (See~\cite[Example 3.25]{DW}.)
\end{ex}

It seems unclear whether, for given $G\subset \Homeo_+(\T)$ one can take as semi-direct product $F\rtimes \Z_m$ a subgroup of $\Homeo_+(\T)$ containing $G$. However we have the following:

\begin{thm}\label{t:RealSemiDirect}
Let $G=F(S_0)\rtimes \Z_m$ be a semi-direct product, where $\Z_m$ acts on the set of generators $S_0$ by permutations.
Then there exists a locally discrete group in $\Diff_+^\omega(\T)$ isomorphic to $G$, acting with an invariant Cantor set.
\end{thm}

\begin{proof}
Denote by $\sigma$ the permutation of $S_0$ defined by a generator of $\Z_m$ and let $\sigma=\gamma_1\cdots \gamma_r$ the cycle decomposition of $\sigma$.
In the following, given $\vartheta\in \T=\R/\Z$, we denote by $R_\vartheta:\T\to \T$ the rotation by $\vartheta$.
Take a partition of $\T$ into $m$ cyclically ordered intervals $I_1,\ldots,I_m$ of equal length. Let $k$ be the length of the cycle $\gamma$.
Choose a generator $f\in S_0$ inside this cycle. Realize $f$ as a diffeomorphism with finitely many fixed points, all hyperbolic, in the following way:
start with an interval, say $I_1$, and choose two disjoint sub-intervals $I_1^-,I_1^+\subset I_1$ (they depend on $\gamma$, but we avoid writing explicitly the dependence); for $i\in \{1,\ldots,m/k\}$ denote by $I_{1+k\,i}^\pm$ the images $R_{i(k/m)}(I_1^{\pm})$, which are contained in $I_{1+k\,i}=R_{i(k/m)}(I_1)$; declare that $f$ has exactly $2m/k$ fixed points, one in each interval $I_{1+k\,i}^\pm$, with those in $I_{1+k\,i}^+$ which are attracting, whereas those in $I_{1+k\,i}^-$ are repelling; we can choose such an $f$ such that 
\begin{equation}\label{eq:f_ping-pong}
f\left (\T\setminus \bigcup_{i=1}^{m/k} I_{1+k\,i}^- \right )=\bigcup_{i=1}^{m/k} I_{1+k\,i}^+,
\end{equation}
and such that the rotation by $k/m$ commutes with $f$.
Next, define $\sigma(f):=R_{1/m}fR_{1/m}^{-1}$: informally speaking, one has to copy what $f$ does to the (right) adjacent intervals of the partition.
Repeat this procedure up to $\sigma^{k-1}(f)$.
Then, realize all the cycles together so that the intervals $I_{1}^\pm$, for different cycles, are pairwise disjoint. (See Figure~\ref{f:RealSemi} for an illustration of the construction in the case of one cycle.)
\end{proof}

\begin{figure}[t]
\[
\includegraphics[scale=1]{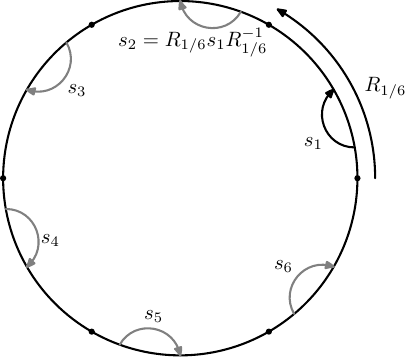}
\]
\caption{Realizing an action of $F_6\rtimes \Z_6$, where a generator of $\Z_6$ acts as a cyclic permutation of the generators of $F_6$.}\label{f:RealSemi}
\end{figure}

\begin{cor}\label{cor:VirtFreeCantor}
Any finitely generated free-by-finite cyclic group which is not virtually cyclic is isomorphic to a locally discrete group in $\Diff_+^\omega(\T)$ acting with a minimal invariant Cantor set.
\end{cor}
\begin{proof}
	Let $G$ be a free-by-cyclic group which is not virtually cyclic. Using Proposition \ref{cor:perm}, $G$ is a subgroup of a semi-direct product $F\rtimes \Z_m$, where $F$ is a finite-rank free group and $\Z_m$ acts by permutations of free generators. By Theorem~\ref{t:RealSemiDirect}, $F\rtimes \Z_m$ is isomorphic to a locally discrete subgroup $\widetilde{G}$ of $\Diff_+^\omega(\T)$ with invariant Cantor set $\widetilde\Lambda$. As $G$ is a subgroup of $F\rtimes \Z_m$, it is isomorphic to a subgroup of $\widetilde{G}$ and hence locally discrete in $\Diff_+^\omega(\T)$. Moreover, as $\widetilde{G}$ does not act minimally, neither can $G$. Finally, as $G$ is not virtually cyclic, Hector's lemma (mentioned in the introduction) implies that $G$ cannot act with a finite orbit, and therefore it also has an invariant Cantor set $\Lambda\subset \widetilde\Lambda$.
\end{proof}

\begin{proof}[Proof of Theorem~\ref{t:RealVirtFree}]
	Direct consequence of Theorem~\ref{t:StructureVirtFree} and Corollary \ref{cor:VirtFreeCantor}.
\end{proof}

\begin{rem}\label{rem:particular_cases}
	Following the arguments appearing in this section for the particular cases of amalgamated products and HNN extensions, one finds the following.
	\begin{itemize}
		\item Let $G\subset \Homeo_+(\T)$ be a virtually free group which decomposes as an amalgam of finite cyclic groups. Then $G\cong \Z_{km}*_{\Z_{k}}\Z_{k\ell}$, for some $k,\ell,m\in\N$, so that $G$ is the central extension
		\[
		0\to\Z_k\to G\to \Z_m*\Z_\ell\to 1.
		\]
		\item Let $G\subset \Homeo_+(\T)$ be a virtually free group which decomposes as an HNN extension of a finite cyclic group $\Z_{m}$. Then there exists an integer $k$ diving $m$ such that $G$ has center isomorphic to $\Z_k$ and the quotient $G/\Z_k$ is isomorphic to the free product $\Z*\Z_{m/k}$, so that  $G$ is the central extension
		\[0\to \Z_k \to G\to \Z*\Z_{m/k}\to 1.\]
	\end{itemize}
\end{rem}

\section{Ping-pong}\label{ping-pong}

\subsection{The statement}

Here we discuss the extension of the classical Klein's ping-pong lemma to graph of groups. Although this is analogous to working with normal forms, we were not able to locate a place in the literature where it is stated in this dynamical form: from the automatic structure of a virtually free group, it is not surprising that such a statement exists, but exhibiting the minimal list of conditions is far from obvious, and probably previously undone. For the particular cases of amalgamated products and HNN extensions, this is originally due to Fenchel and Nielsen (the now-published notes \cite{Fenchel-Nielsen} were circulating in the 1950s) and it appears in detail in the work of Maskit \cite[Chapter VII]{Maskit}. Observe that for general graphs of groups, our  ping-pong lemma is difficult to use, as one needs efficient ways of solving the membership problem for $\alpha_e(A_e)\subset G_{o(e)}$ (cf.~\cite{kapovich2005foldings}).

The reader who is not familiar with Bass--Serre theory could find helpful to read first \S\S\ref{sc:arboreal_part} and \ref{sc:arboreal_ping-pong}. In particular, the conditions listed in the following fundamental definition are very natural when considering the action on a Bass--Serre tree.

\begin{dfn}\label{d:generalized_ping-pong}
Consider a graph $\overline X=(V,E)$ and let $G=\pi_1(\overline X;G_v,A_e)$ be the fundamental group of a graph of groups.
Choose a spanning tree $T=(V,E_T)\subset \overline X$, and let $S=E\setminus E_T$ be the collection of oriented edges not in $T$. We denote by $\mathsf{St}_T(v)=\{e\in E_T\mid o(e)=v\}$ the \emph{star} of $v$ in $T$;  given $e\in E_T$ and $v\in V$, we also write $v\in C(e,T)$ if the edge $e$ belongs to the oriented geodesic path in $T$ connecting $o(e)$ to $v$.

Given an action of the group $G$ on a set $\Omega$, a family of subsets $\boxminus=\{X_v,Z_s\}_{v\in V,s\in S}$ is called an \emph{interactive family}\footnote{We use the symbol $\boxminus$ for it as it reminds a ping-pong table. Also, we use the terminology of interactive family, in analogy with \emph{interactive pair} and \emph{interactive triple} appearing in Maskit \cite{Maskit}. In this work we reserve the word \emph{ping-pong} for partitions of the circle.} if:
\begin{enumerate}[(\text{IF} 1)]

\item the subsets $X_v$, $Z_s$ (for $v\in V$ and $s\in S$) are pairwise disjoint sets; the subsets $Z_s$ are non-empty, and if $G_v\neq \alpha_e(A_e)$ for some edge $e\in E$ such that $o(e)=v$, then $X_v\neq \emptyset$;
\label{pp1}

\item for every edge $s\in S$ and element $O\in\boxminus\setminus \{Z_{\overline s}\}$, one has $s(O)\subset Z_s$; \label{pp2}

\item  the inclusion $\alpha_s(A_s)(Z_s)\subset Z_s$ holds for every $s\in S$;\label{pp3}

\item the inclusion $(G_{o(s)}\setminus \alpha_s(A_s))(Z_s)\subset X_{o(s)}$ holds for every $s\in S$;\label{pp4}

\item for $v\in V$ and $e\in \mathsf{St}_T(v)$ such that $G_v\neq \alpha_e(A_e)$, the corresponding $X_v$ contains a non-empty subset $X_v^e$;\label{pp5}

\item the inclusion $\alpha_e(A_{e})(X_{o(e)}^e)\subset X^e_{o(e)}$ holds for every $e\in E_T$; \label{pp6}

\item if $e\in E_T$ and $v\in V$ are such that $v\in C(e,T)$, then $\left (G_{o(e)}\setminus \alpha_e(A_e)\right )(X_{v})\subset X^e_{o(e)}$ (in particular this holds for $v=t(e)$);\label{pp7} 

\item if $e\in E_T$ and $s\in S$ are such that $o(s)\in C(e,T)$, then $\left (G_{o(e)}\setminus \alpha_e(A_e)\right )(Z_s)\subset X^e_{o(e)}$.\label{pp8}

\end{enumerate}

In addition, one says that the interactive family is \emph{proper} if the following holds:

\begin{enumerate}[(\text{IF} 1)]
\setcounter{enumi}{8}

\item for $s\in S$, the restriction of the action of $\alpha_s(A_s)$ to $Z_s$ is faithful, and similarly, for $e\in E_T$,  the restriction of the action of $\alpha_e(A_e)$ to $X_{o(e)}^e$ is faithful;\label{pp9}

\item if there exists a vertex $v\in V$ such that $X_v\neq\emptyset$, then there exists a (possibly different) vertex $w\in V$ such that the union of all the images from (IF~\ref{pp4},\ref{pp7},\ref{pp8}) inside the corresponding $X_{w}$ misses a point;
\label{pp10}

\item if $S=\{s,\overline s\}$ and $X_v=\emptyset$ for every $v\in V$, then we require that there exists a point $p\in \Omega\setminus(Z_s\cup Z_{\overline{s}})$ such that $s(p)\in Z_s$ and $\overline s(p)\in Z_{\overline s}$.\label{pp11}
\end{enumerate}
\end{dfn}

\begin{rem}\label{pipipi}
	Condition (IF~\ref{pp11}) is an \textit{ad hoc} condition introduced to cover the degenerate case $G\cong G_0\rtimes \Z$ (for some $G_0$). This case never occurs in our applications to dynamics of virtually free groups on the circle, as such semi-direct product gives a virtually cyclic group. We will use this condition explicitly only in Lemma \ref{l:pp_HNN}.
\end{rem}

\begin{rem}
	Observe that conditions (IF~\ref{pp3},\ref{pp6}) actually imply (whence they are equivalent to) the corresponding equalities $\alpha_s(A_s)(Z_s)=Z_s$ and $\alpha_e(A_{e})(X_{o(e)}^e)= X^e_{o(e)}$.
\end{rem}

The next two subsections are devoted to the proof of the following:

\begin{thm}
	[Generalized ping-pong lemma]
	\label{p:faithful}
	Consider a graph $\overline X=(V,E)$ and let $G=\pi_1(\overline X;G_v,A_e)$ be the fundamental group of a graph of groups.
	Choose a maximal subtree $T=(V,E_T)\subset \overline X$, and let $S=E\setminus E_T$ be the collection of oriented edges not in $T$.
	
	Then any action of $G$ on a set $\Omega$, which admits a proper interactive family for $(G_v,A_{e})$, is faithful.
\end{thm}

\subsection{Basic cases}\label{sc:basic}

First we discuss the two particular cases appearing in \cite{Maskit}, whose proofs basically rely on the normal forms for elements in amalgamated products and HNN extensions.

\begin{lem}[Ping-pong for amalgamated products; see \S VII.A.10 of \cite{Maskit}]
	\label{l:pp_amalgam}
	Assume that the graph $\overline X$ consists of a single edge $e$ between two distinct vertices $o(e)$ and $t(e)$. Let $G=\pi_1(\overline X;G_v,A_e)\cong G_{o(e)}*_{A_e}G_{t(e)}$ be the fundamental group of a graph of groups.
	
	Then any action of $G$ on a set $\Omega$, which admits a proper interactive family, is faithful.
\end{lem}

\begin{proof}
	For simplicity, we write $G_1=G_{o(e)}$, $G_2=G_{t(e)}$ and $A=\alpha_e(A_e)=\omega_e(A_e)\subset G_1\cap G_2$.
	In this case we have that $\overline X$ is a tree, so $S=\emptyset$. 
	Observe that if $G_2=A$, then $G\cong G_1$, so that the statement follows trivially from (IF~\ref{pp7},\ref{pp9}). Therefore we can assume $G_1,G_2\neq A$.
	Conditions in Definition~\ref{d:generalized_ping-pong} reduce to the following:
	\begin{enumerate}[(IF 1)]
		\item[(IF \ref{pp1})]$\Omega$ contains two non-empty disjoint subsets $X_1,X_2$;
		\item[(IF \ref{pp5})] there are subsets $X_1^e\subset X_1$, $X_2^{\overline e}\subset X_2$;
		\item[(IF \ref{pp6})] $A(X^{e}_1)\subset X^{e}_1$ and $A(X_2^{\overline e})\subset X^{\overline e}_2$;
		\item[(IF \ref{pp7})] $(G_1\setminus A)(X_2)\subset X_1^e$ and 
		$(G_2\setminus A)(X_1)\subset X_2^{\overline e}$;
		\item [(IF \ref{pp9})] the restrictions of the action of $A$ to $X_{1}^e$ and $X_2^{\overline e}$ respectively are both faithful;
		\item [(IF \ref{pp10})] at least one of the inclusions $(G_1\setminus A)(X_2)\subset X_1$ or $(G_2\setminus A)(X_1)\subset X_2$, is proper.
	\end{enumerate}
	These conditions are close to that of \emph{proper interactive pair} in  \cite{Maskit}, with minor differences. The lines of the proof in \cite{Maskit} work, but we detail for clarity.
	If $g\in G_1\cup G_2$ is not the identity element, then (IF~\ref{pp7},\ref{pp9}) guarantee that $g$ does not act trivially. 
	Without loss of generality, we can assume from (IF \ref{pp10}) that there exists a point $x\in X_1$ which is not in $(G_1\setminus A)(X_2)$.
	If $g\in G$ is not conjugate into $G_1\cup G_2$, then, up to make a conjugacy, we can write $g$ as a product
	$g=g_nh_n\cdots g_1h_1$, with $g_i\in G_1\setminus A$ and $h_i\in G_2\setminus A$ (normal form for amalgamated free products \cite{LScombinatorial}).
	Then using (IF \ref{pp7}) we get
	\[h_1(x)\in X_2^{\overline e},\quad g_1h_1(x)\in X_1^e,\quad\ldots,\quad h_n\cdots g_1h_1(x)\in  X_2^{\overline e}\]
	and finally $g(x)=g_nh_n\cdots g_1h_1(x)\in (G_1\setminus A)(X_2^{\overline e})$. As $x$ is not in  $(G_1\setminus A)(X_2^{\overline e})$, we must have $g(x)\neq x$, and thus $g$ does not act trivially on $\Omega$.
\end{proof}

\begin{lem}[Ping-pong for HNN extensions; see \S VII.D.12 of \cite{Maskit}]
	\label{l:pp_HNN}
	Assume that the graph $\overline X$ consists of a single vertex $0$ with a self-edge $s$. Let $G=\pi_1(\overline X;G_0,A_s)\cong G_{0}*_{A_s}$ be the fundamental group of a graph of groups.
	
	Then any action of $G$ on a set $\Omega$, which admits a proper interactive family, is faithful.
\end{lem}

\begin{proof}
	For simplicity, we write $A=A_s$.
	First we rule out the case $\alpha_s(A)=\omega_s(A)=G_0$, for which we introduced condition (IF~\ref{pp11}). Then $G\cong G_0\rtimes \Z$ and every element $g\in G$ is uniquely represented by a product $g_0 s^n$, $g\in G_0$, $n\in\Z$. If $n=0$, then by (IF \ref{pp9}) the element $g=g_0$ cannot act trivially, unless $g=\mathrm{id}$.
	Otherwise, if $n\neq 0$, take the point $p\in \Omega\setminus(Z_s\cup Z_{\overline s})$ given by (IF~\ref{pp11}). Then $s^n (p)\in Z_s\cup Z_{\overline{s}}$ and $g(p)=g_0s^n(p)\in Z_s\cup Z_{\overline{s}}$ as well, for $G_0=\alpha_s(A)=\omega_s(A)$ preserves both $Z_s$ and $Z_{\overline s}$ after (IF \ref{pp3}). Thus $g(p)\neq p$.
	
	Since now we can assume $\alpha_s(A)\neq G_0$.
	The vertex $0$ is the ``spanning tree'' in $\overline X$ and $S=\{s,\overline s\}$. Conditions in Definition~\ref{d:generalized_ping-pong} reduce to:
	\begin{enumerate}[(IF 1)]

		\item[(IF \ref{pp1})] $\Omega$ contains non-empty disjoint subsets $X_0,Z_s,Z_{\overline s}$;

		\item[(IF \ref{pp2})] $s(X_0\cup Z_{s})\subset Z_s$ and $s^{-1}(X_0\cup Z_{\overline s})\subset Z_{\overline s}$;

		\item[(IF \ref{pp3})] $\alpha_s(A)(Z_s)\subset Z_s$ and $\omega_s(A)(Z_{\overline s})\subset Z_{\overline s}$; 

		\item[(IF \ref{pp4})] $(G_0\setminus \alpha_s(A))(Z_s)\subset X_0$ and 
		$(G_0\setminus \omega_s(A))(Z_{\overline s})\subset X_0$;
		
		\item[(IF \ref{pp9})] both the restrictions of the actions of $\alpha_s(A)$ to $Z_s$ and of $\omega_s(A)$ to $Z_{\overline s}$ respectively, are faithful;

		\item[(IF \ref{pp10})] the union of the images 
		$(G_0\setminus \alpha_s(A))(Z_s)\cup
		(G_0\setminus \omega_s(A))(Z_{\overline s})$ misses a point in $X_0$.
	\end{enumerate}
	In the language of \cite{Maskit}, $(X_0,Z_{s},Z_{\overline s})$ is a \emph{proper interactive triple} for $G$, which guarantees that the action on $\Omega$ is faithful. (Note that there is however a minor difference compared to the statement in \cite{Maskit}, as here we require additionally in (IF \ref{pp9}) that edge groups act faithfully; this is because at the end we want to ensure that the action on $\Omega$ is faithful.)
\end{proof}

\begin{lem}[Ping-pong for generalized  HNN extensions]
	\label{l:pp_multipleHNN}
	Assume that the graph $\overline X$ consists of a single vertex $0$ with $n$ self-edges $s_1,\ldots,s_n$. Let $G=\pi_1(\overline X;G_0,A_s)$ be the fundamental group of a graph of groups.
	
	Then any action of $G$ on a set $\Omega$, which admits a proper interactive family, is faithful.
\end{lem}

\begin{proof}
	As in the proof of Lemma~\ref{l:pp_HNN}, the case $\alpha_s(A_s)=G_0$ for every $s\in S$ must be treated separately, as in this case $G\cong G_0\rtimes F_n$. As the same proof as before works, with minor changes, we pass directly to the case $\alpha_s(A_s)\neq G_0$ for some $s\in S$. 
	If $n=0$ there is nothing to prove, and if $n=1$ the statement is the one in Lemma~\ref{l:pp_HNN}. 	
	So we assume $n>1$ and proceed by induction.
	As in the proof of Lemma \ref{l:pp_HNN}, the vertex $0$ is the ``spanning tree'' in $\overline X$ and $S=\{s_i,\overline s_i\}$. Also, only conditions (IF \ref{pp1}--\ref{pp4},\ref{pp9},\ref{pp10}) say something non trivial about the interactive family, and they reduce to the conditions listed in the proof of Lemma \ref{l:pp_HNN}, for every $s_1,\ldots,s_n$, with (IF \ref{pp1}) requiring additionally that all the subsets $X_0,Z_{s_i},Z_{\overline s_i}$ are pairwise disjoint, (IF~\ref{pp2}) dictating further inclusions, and (IF \ref{pp10}) requiring that the union of all images miss a point in $X_0$.
	In particular, by Lemma \ref{l:pp_HNN}, we have that the action restricted to $G_1:=\langle G_0,s_1\rangle\cong G_0*_{A_{s_1}}$ is faithful.
	
	Then we consider a new graph of groups $(\widehat X; G_1, A_s)$, where the graph $\widehat X$ is obtained from $\overline X$ by contracting the edge $s_1$, $G_1$ is the new vertex group, and all the edge groups remain the same, with boundary injections $\alpha_{s_i}:A_{s_i}\to G_1$, $\omega_{s_i}:A_{s_i}\to G_1$ obtained from the previous ones by post-composition with the natural injection $G_0\subset G_1$.
	
	Consider the new subset $X_1:=X_0\cup Z_{s_1}\cup Z_{\overline s_1}\subset \Omega$. We want to verify that $\boxminus'=\{X_1,Z_s\}_{s\neq s_1,\overline s_1}$ defines a proper interactive family for the new graph of groups. Again, we need only verify conditions (IF \ref{pp1}--\ref{pp4},\ref{pp9},\ref{pp10}). Observe that (IF \ref{pp1}) is easily verified, and for (IF \ref{pp2},\ref{pp3},\ref{pp9}) there is no new condition to verify. So we only need to check (IF \ref{pp4},\ref{pp10}).
	
	For (IF \ref{pp4}) we want to prove that for every $s\neq s_1,\overline s_1$, one has
	$(G_1\setminus \alpha_s(A_{s}))(Z_s)\subset X_1$.	
	Take $g\in G_0\setminus \alpha_s(A_{s})\subset G_1\setminus \alpha_s(A_{s})$; in this case $g(Z_s)\subset X_0\subset X_1$ by the old condition (IF \ref{pp4}) for $\boxminus$.	
	Next, take an element $g\in G_1\setminus G_0\subset G_1\setminus \alpha_s(A_{s})$ and write
	\begin{equation}\label{eq:normal_HNN}
	g= g_ks_1^{\epsilon_k}\cdots g_1s_1^{\epsilon_1}g_0,
	\end{equation}
	where $k\ge 1$, $\epsilon_j\in \{\pm 1\}$, $g_j\in G_0$, where if $g_j\in \alpha_{s_1}(A_{s_1})$ and $\epsilon_j=1$ then $\epsilon_{j+1}=1$, while if  $g_j\in \omega_{s_1}(A_{s_1})$ and $\epsilon_j=-1$ then $\epsilon_{j+1}=-1$ (normal form for HNN extensions \cite{LScombinatorial}). If $g_0\notin \alpha_s(A_s)$, then $g_0(Z_s)\subset X_0$ by the old condition (IF \ref{pp4}) for $\boxminus$. Otherwise $g_0(Z_s)\subset Z_s$ by the old condition (IF \ref{pp3}) for $\boxminus$. In both cases, $g_0(Z_s)$ does not intersect $Z_{s_1}\cup Z_{\overline s_1}$, so that applying the following letter in the expression \eqref{eq:normal_HNN} for $g$, one has $s_1^{\epsilon_1}g_0(Z_s)\subset Z_{s_1}\cup Z_{\overline s_1}$ (because of the old condition (IF \ref{pp2}) for $\boxminus$). Proceeding in this way (see~\cite[\S VII.D.11]{Maskit}), one has $g(Z_s)\subset X_1$, as desired.
	
	Finally, let us verify (IF \ref{pp10}).
	Observe that by (IF \ref{pp3}), the old condition (IF \ref{pp10}) for $\boxminus$ can be reformulated as
	\[
	\bigcup_{s\in S}G_0(Z_{s})\cap X_0\subset X_0\text{ is proper};
	\]
	the discussion in the previous paragraph gives the equality for every $s\neq s_1,\overline s_1$:
	$G_1(Z_{s})\cap X_0 = G_0(Z_{s})\cap X_0$.
	Thus one has that the inclusion
	\[
	\bigcup_{s\in S\setminus \{s_1,\overline s_1\}}G_1(Z_{s})\cap X_0\subset X_0\text{ is proper},
	\]
	which readily implies that the inclusion
	\[
	\bigcup_{s\in S\setminus \{s_1,\overline s_1\}}G_1(Z_{s})\cap X_1\subset X_1\text{ is proper}.\qedhere
	\]
\end{proof}

\subsection{Proof of Theorem~\ref{p:faithful}}\label{sc:ping-pong}

Observe first that if $X_v=\emptyset$ for every $v\in V$, then all edge groups must be isomorphic and coincide with every vertex group. Therefore one can contract the spanning tree $T$ to a single vertex, and the group $G$ is isomorphic to a semi-direct product $G_0\rtimes F_n$, already covered in Lemma~\ref{l:pp_multipleHNN}. From now on we assume $X_v\neq \emptyset$ for some $v\in V$. 

We proceed by induction on the number of edges of a spanning tree $T$ of the graph $\overline{X}$ (the number of edges of a spanning tree does not depend on the spanning tree). If there is no edge in $T$, then we are in the case covered by Lemma~\ref{l:pp_multipleHNN}. So let $\overline{X}=(V,E)$ be a graph with a marked spanning tree $T=(V,E_T)$ with at least one edge, $(\overline X; G_v,A_e)$ a graph of groups, and $G$ its fundamental group. Let us fix an edge $e_0\in E_T$ such that the corresponding inclusions in (IF \ref{pp10}) miss a point for $o(e_0)$, and write for simplicity $1=o(e_0)$, $2=t(e_0)$, $A=\alpha_{e_0}(A_{e_0})=\omega_{e_0}(A_{e_0})\subset G_1\cap G_2$. As in Lemma~\ref{l:pp_amalgam}, conditions (IF \ref{pp1},\ref{pp5}--\ref{pp7},\ref{pp9},\ref{pp10}) guarantee that the action restricted to $G_{0}:=\langle G_{1},G_{2}\rangle\cong G_{1}*_{A}G_{2}$ is faithful.

As in the proof of Lemma~\ref{l:pp_multipleHNN}, we define a new graph $\widehat X=(\widehat{V},\widehat E)$ obtained from $\overline X$ by contracting the edge $e_0$: $\widehat V= V/_{1\sim 2}\cong (V\setminus \{1,2\})\cup\{0\}$, $\widehat E=E\setminus \{e_0,\overline e_0\}$. In the following, we write $V_0:=V\setminus \{1,2\}\subset \widehat V$. We also denote by $\widehat T$ the image of the spanning tree $T$ under contraction, and we note that the set $S$ of edges in the complement of $E_T$ is not affected by contraction, so that we still denote by $S$ the complement $E_{\widehat T}$ in $\widehat E$. Then we define a new graph of groups $(\widehat X; G_v,A_e)$, where all the vertex groups $G_v$, $v\in V_0$, and edge groups $A_e$, $e\in \widehat E$, remain the same, while $G_0$ is the vertex group of the vertex $0$, which is the image of $e_0$ under contraction. Observe that boundary injections are naturally defined, as $G_{1}$ and $G_{2}$ are naturally identified as subgroups of $G_0\cong G_{1}*_{A}G_{2}$.
We define $X_0:=X_{1}\cup X_{2}$. We want to verify that the new family $\boxminus'=\{X_0,X_v,Z_s\}_{v\in V_0,s\in S}$ is an interactive family for the new graph of groups. Condition (IF \ref{pp1}) is clear. Since we have not modified what depends on $S$, conditions (IF \ref{pp2},\ref{pp3}) follow from the old ones for $\boxminus$.

Let us verify condition (IF \ref{pp4}). If $s\in S$ is such that $o(s)\neq 1,2$ in $\overline X$, then this reduces to the old condition (IF \ref{pp4}) for $\boxminus$. So let us assume $o(s)\in \{1,2\}$, and without loss of generality we assume $o(s)=1$. Then condition (IF \ref{pp4}) follows from the normal form in amalgamated products \cite{LScombinatorial} and the old (IF~\ref{pp3},\ref{pp4},\ref{pp7},\ref{pp8}) for $\boxminus$. More precisely, take an element $g\in G_0\setminus \alpha_s(A_s)$ (observe that $\alpha_s(A_s)\subset G_1\subset G_0$). If $g\in G_1\setminus \alpha_s(A_s)$, there is nothing new to verify. If $g\in G_2\setminus \alpha_s(A_s)=G_2\setminus A$, then $g(Z_s)\subset X_2^{\overline e_0}\subset X_0$ after the old condition (IF \ref{pp8}) for $\boxminus$.
  Thus we can assume $g\in G_0\setminus G_1$ and write $g$ as a normal form
\[
g=g_kh_k\cdots g_1h_1g_0 \quad \text{or}\quad g=h_k\cdots g_1h_1g_0,
\]
where $k\ge 1$, $g_0\in G_1$, $g_i\in G_1\setminus A$, $h_i\in G_2\setminus A$ for $i\ge 1$. If $g_0\in \alpha_s(A_s)$, then $g_0(Z_s)=Z_s$ by (IF \ref{pp3}) for $\boxminus$, otherwise $g_0(Z_s)\subset X_{1}$  by (IF \ref{pp4}) for $\boxminus$. In both cases, when applying $h_1\in G_2\setminus A$, one has $h_1(Z_s\cup X_1)\subset X_2^{\overline e_0}$ because of (IF \ref{pp7},\ref{pp8}) for $\boxminus$. Therefore $h_1g_0(Z_s)\subset X_2^{\overline e_0}$. Proceeding in this way, one ends up with $g(Z_s)\subset X_0$.

Next, let us introduce the subset
\[
\Omega_1:=\left (G_1\setminus A\right )\left (\bigcup_{v\in C(e_0,T)\setminus \{1,2\}} X_v\cup \bigcup_{s\in S\,:\,o(s)\in C(e_0,T)\setminus \{1,2\}} Z_s\right )
\]
where,  recalling our notation $C(e_0,T)$, the first union is taken over vertices $v\in V_0$ such that $e_0$ is in the geodesic path from $1$ to $v$ in $T$, and similarly the second one is defined over edges $s\in S$ such that $e_0$ is in the geodesic path from $1$ to $o(s)$ in $T$, and $o(s)\in V_0$. By the old conditions (IF~\ref{pp7},\ref{pp8}) for $\boxminus$, we have $\Omega_1\subset X_1^{e_0}$.
Similarly, we introduce the subset
\[
\Omega_2:=\left (G_2\setminus A\right )\left (\bigcup_{v\in C(\overline e_0,T)\setminus \{1,2\}} X_v\cup \bigcup_{s\in S\,:\,o(s)\in C(\overline e_0,T)\setminus \{1,2\}} Z_s\right )
\]
where unions are taken over vertices $v\in V_0$ and edges $s\in S$ verifying analogous properties with respect to $\overline e_0$. Then $\Omega_2\subset X_2^{\overline e_0}$ by the old conditions (IF \ref{pp7},\ref{pp8}) for $\boxminus$.

For $e\in E_{\widehat T}$ such that $o(e)=0$ in $\widehat X$ (that is, $o(e)\in \{1,2\}$ in $\overline X$), we introduce subsets $X_0^e\subset X_0$.
Let us also introduce some extra notation. Let $i\in\{1,2\}$ be the vertex such that $o(e)=i$ in $\overline X$ and the vertex $i'$ be the one in $\{1,2\}$ different from $i$; for $k\ge 0$, define
\begin{align*}
Y'_k&=\left ((G_{i'}\setminus A)(G_{i}\setminus A)\right )^k(\Omega_{i'}),\\
Y_k&=(G_{i}\setminus A)\left ((G_{i'}\setminus A)(G_{i}\setminus A)\right )^k(\Omega_{i'})=(G_{i}\setminus A)(Y_k'),
\end{align*}
which are respectively subsets of $X_{i'}^{f}$ and $X_i^{\overline f},$ where $f\in\{e_0,\overline e_0\}$ is such that $i'=o(f)$. Then we define
\[
X^e_0 :=
X_{i}^e\cup\bigcup_{k\ge 0}Y'_k\cup  Y_k,\]
which is a subset of $X_i^e\cup X_1^{e_0}\cup X_2^{\overline e_0}\subset X_0$, as required for (IF \ref{pp5}). We have to verify that this is the good choice for the other conditions.

Let us check (IF \ref{pp6}). The only new condition that we have to prove is that if an edge $e$ in $\widehat X$ satisfies $o(e)=i\in\{1,2\}$ in $\overline X$, then $\alpha_e(A_e)(X_0^e)\subset X_0^e$.
For this,  we observe that after the old condition (IF \ref{pp6}) for $\boxminus$, $\alpha_e(A_e)(X_i^e)=X_i^e$ and the corresponding restriction of the action is faithful, after (IF \ref{pp9}) for $\boxminus$.
This immediately implies that the restriction of the action of $\alpha_e(A_e)$ to $X_0^e$ is also faithful, establishing (IF~\ref{pp9}).
Take now an element $a\in \alpha_e(A_e)\setminus A\subset G_i\setminus A$, then, working with the definitions of $Y_k$ and $Y_k'$, one verifies that $a (Y'_k)\subset Y_{k}$ and $a(Y_k)\subset Y_k\cup Y'_{k}$; similarly, if $a\in \alpha_e(A_e)\cap A\subset G_1\cap G_2$, then $a(Y_k)\subset Y_k$ and $a(Y'_k)\subset Y'_k$ (in this last case, when $k=0$, one has to use the definition of $\Omega_{i'}$ as an image by $G_{i'}\setminus A$, which is $A$-invariant).  This gives (IF \ref{pp6}).

Let us check (IF \ref{pp7}). There are two cases where the old condition (IF \ref{pp7}) for $\boxminus$ does not apply directly. First, assume that $v=0$, and $e\in E_{\widehat T}$ is in the geodesic path in $T$ from $o(e)$ to $v$ (using our notation, we are assuming $0\in C(e,\widehat T)$). Observe that if this happens for $v=1$ in $\overline X$ then this happens also for $v=2$ in $\overline X$ (because we can concatenate the geodesic path with $e_0$, if it does not contain $e_0$ already), and vice versa.
Then the old condition (IF \ref{pp7}) for $\boxminus$ gives that $\left (G_{o(e)}\setminus \alpha_e(A_e)\right )(X_1\cup X_2)\subset X_{o(e)}^e$, as wanted.
Second, assume that $v\in V_0$, and $o(e)=0$. 
Then we argue as for (IF \ref{pp4}) using the normal form in amalgamated products, and we use the definition of the subsets $Y_k,Y_k'$ (that we made on purpose)
to get $\left (G_0\setminus \alpha_e(A_e)\right )(X_v)\subset X_0^e$.
Similarly one proves condition (IF \ref{pp8}).

It remains to verify condition (IF \ref{pp10}). 
We want to prove that for the new vertex $0\in \widehat V$, the union of all the images from the new (IF~\ref{pp4},\ref{pp7},\ref{pp8}) for $\boxminus'$ inside $X_0=X_1\sqcup X_2$ misses a point. As we have started with old condition (IF \ref{pp10}) for $\boxminus$ being satisfied by the vertex $1\in V$, we will simply look at the intersection of such new images with $X_1\subset X_0$. But then, using the normal form as before, we get that the new images in $X_1$ under elements in $G_0$ are contained in images in $X_1$ under elements of $G_1$, so the new (IF \ref{pp10}) for $\boxminus'$ follows from the old condition (IF \ref{pp10}) for $\boxminus$.\qed

\subsection{The arboreal partition}\label{sc:arboreal_part}\label{ssc.dyn_arb}

In this section we explain how interactive families are naturally defined 
for proper actions of a group on a tree. In fact, the list of conditions in Definition~\ref{d:generalized_ping-pong} of interactive family have been deduced by looking carefully at the properties enjoyed by these natural partitions on trees.
We do not claim originality for this part: experts in Bass--Serre theory know this matter very well, and it appears in many places in the literature, with slight modifications (see for instance \cite[\S 12.3]{hyperbolic} or \cite[\S 3.2]{automatic}). Indeed, typically one looks at the partition induced on the \emph{boundary} of the tree.

Consider an orientation-preserving proper action $\alpha:G\to \Isom_+(X)$ of a group $G$ on a tree $X$. In the following we shall only consider \emph{right actions}.
Let $T=T^G\subset X$ be a connected fundamental domain for the action $\alpha$.
By such a notion, we mean here that $T$ is a minimal connected  subset containing exactly a point of every orbit of the action on the simplicial tree $X$. Furthermore, we require that the boundary $\partial T$ consists only of vertices of the tree ($T$ may fail to be closed).
We denote by $\{X_c=X_c^G\}_{c\in \pi_0(X\setminus T)}$ the collection of 
connected components of the complement of $T$ in $X$. When $X_c$ is adjacent to a vertex $v\in T$, we denote by $e_c$ the edge of $X_c$ which is adjacent to this vertex, oriented outwards, namely $o(e_c)=v$.
Otherwise, there is a unique vertex $v_c\in X_c$ adjacent to $T$. We let $\delta_c$ denote either the oriented edge $e_c$ or the vertex $v_c$, accordingly.
By the moment we do not put any further restriction on the action (and on the group).

For any $c\in \pi_0(X\setminus T)$, we define
\[W_c=W_c^G:=\{g\in G\mid T.g\subset X_c\}.\]
Let $G^\circ$ be the set of elements in $G$ which fix some vertex in $T$:
\[
G^\circ:=\{g\in G\mid \text{$g$ fixes a vertex in $T$}\}.
\]
As the action has finite stabilizers, the set $G^\circ$ is finite.
We have a partition
\[
G=G^\circ \sqcup \bigsqcup_{c\in \pi_0(X\setminus T)}W_c.
\]
In the rest of this section we study how this partition, which we call the \emph{arboreal partition}, behaves under the regular action of the group $G$ on the right.

\begin{lem}
	For any connected component $c,d\in \pi_0(X\setminus T)$ and element $g\in G$ such that $X_c.g\subset X_d$, then
	\[
	W_cg\subset W_d.
	\]
\end{lem}
\begin{proof}
	If $h\in W_c$ then $T.h\subset X_c$.  The condition implies $T.hg\subset X_d$, as wanted.
\end{proof}

\begin{lem}
	Take connected components $c,d\in \pi_0(X\setminus T)$ and an element $g\notin G^\circ$. Then
	\[
	X_c.g\subset X_d \quad\text{if and only if}\quad g\in W_d\text{ and }g^{-1}\notin W_c.
	\]
\end{lem}
\begin{proof}
	The condition that $g\in W_d$ and $g^{-1}\notin W_c$ is equivalent to 
	\[T.g\subset X_d\quad{\text{and}}\quad T\not\subset X_c.g.\]
	Since $T\cap T.g=\emptyset$, the condition $T\not\subset X_c.g$ is equivalent to $T\cap X_c.g=\emptyset$ (because $X_c.g$ is a connected component of $X\setminus T.g$). As $X_c.g$ is adjacent to $T.g$ and $T.g\subset X_d$, this is equivalent to $X_c.g\subset X_d$, as desired.
\end{proof}

\begin{lem}\label{l:arboreal_vertex}
	Take connected components $c, d\in \pi_0(X\setminus T)$ and an element $g\in G^\circ$. Then 
	\[
	X_c.g\subset X_d\quad\text{if and only if}\quad
		\delta_c.g\text{ is in } X_d.
	\]
\end{lem}
\begin{proof}
	As $\delta_c$ is in $X_c$, the nontrivial direction is the ``if'' part.
	An element in $G^\circ$ is an elliptic isometry of the tree: the connected component $X_c$ turns (possibly stays fixed) around $T\cap T.g$.
	Consider the concatenation $\gamma_g$ of the (possibly empty) geodesic path from  $T\cap T.g$ to $X_c$ with $\delta_c$. When this is a single point, then $\delta_c=v_c$, and we have $\delta_c.g=\delta_c$ and $X_c.g=X_c$. Otherwise, the action of $g$ preserves the orientation of $\gamma_g$ (because $g$ is an elliptic isometry of the tree). Thus, if $\delta_c.g$ is in $X_d$, then $X_c.g$ cannot intersect $T$, so $X_c.g\subset X_d$.
\end{proof}

The above lemmas easily imply the following:

\begin{prop}\label{p:inclusionW}Take connected components $c, d\in \pi_0(X\setminus T)$ and an element $g\in G$. The following conditions are equivalent.
	\begin{enumerate}[1.]
		\item $W_cg\subset W_d$.
		\item $X_c.g\subset X_d$.
		\item One has
		\begin{enumerate}[(1)]
			\item \label{i:inclusion1} either $g\in W_d$ and  $g^{-1}\notin W_c$, or
			\item \label{i:inclusion2} $g\in G^\circ$ and $\delta_c.g$ is in $X_d$.
		\end{enumerate}
	\end{enumerate}
\end{prop}

To summarize the content of this proposition, we introduce the following sets: 

\begin{dfn}
	Given two connected components $c,d\in \pi_0(X\setminus T)$, not necessarily distinct, we define
	\[
	W_{c\rightarrow d}:=\left \{g\in G\,\middle\vert\, W_cg^{-1}\subset W_d\right \}.
	\]
\end{dfn}

Thus, by Proposition~\ref{p:inclusionW}, the set $W_{c\rightarrow d}$ coincides with
\[
\left \{ g\in G \,\middle\vert\, g^{-1}\in W_d,\,g\notin W_c \right \}
\cup \left \{g\in G^\circ \,\middle\vert\, \delta_c.g^{-1}\text{ is in } X_d
\right \}.
\]

\begin{rem}
	The sets $W_{c\rightarrow d}$ are not defining a partition of the group. For example the identity is in every $W_{c\rightarrow c}$ as well as any other element in $G$ acting trivially on the tree.
\end{rem}

\begin{rem}
	The definition of $W_{c\rightarrow d}$ may appear counterintuitive because of the presence of inverses. We will see that it is motivated by Lemma~\ref{l:top-alg}, which makes the connection between the action of the group $G$ on the circle and that on a tree $X$. This connection changes the side of the action: since we are interested in studying the dynamics on the circle (for which we use the classical convention of left actions), this morally explains why we want
	$W_cg^{-1}\subset W_d$ if $g\in W_{c\rightarrow d}$.
\end{rem}

\subsection{Arboreal ping-pong}\label{sc:arboreal_ping-pong} 
Here we verify that from an action on a tree we can define an interactive family, in the sense of Definition \ref{d:generalized_ping-pong} (but for right actions!).
For this we maintain the notation introduced in the previous subsection.
That is, we have an action $\alpha:G\to\Isom_+(X)$ of the group $G$ on a tree $X$, and we choose a connected fundamental domain $T$. This gives a presentation of $G$ as the fundamental group of the graph of groups $(\overline X; G_v,A_e)$, as we now describe. The set of complementary edges $S$ correspond to the edges of the fundamental domain $T$ which only have one endpoint in $T$; denote by $\overline T$ the injective projection of $T\setminus S$ to $\overline X$, which is a spanning tree. The groups $G_v$ and $A_e$ are clearly defined as stabilizers for the action $\alpha:G\to \Isom_+(X)$, and so are defined the boundary injections. We will write $\overline X= (V,E)$ as usual. Observe that the set of vertices $V$ lifts to the set of vertices of the fundamental domain $T$ and similarly does the set of edges $E_{\overline T}$.

The interactive family will be defined for the action on $X$. The family of subsets $\boxminus=\{X_v,Z_s\}$ will be defined taking unions of subsets of the form $X_c\subset X$, $c\in\pi_0(X\setminus T)$.

\begin{itemize}
	\item Given a complementary edge $s$, oriented such that $o(s)$ is a vertex of $\overline T$ (and hence of $T$ in the lift to $X$), the associated stable letter (still denoted by $s$) acts on $X$ as a hyperbolic isometry, and hence has an oriented translation axis $X(s)$ which contains the oriented edge $s$; the complement $X(s)\setminus T$ has two connected components, and the orientation of $X(s)$ dictates which component is positive, and which is negative; we denote them respectively by $X^+(s)$ and $X^-(s)$. Then, we 	
	define $Z_s$ as the connected component of $X\setminus T$ which contains $X^+(s)$
	(as $X^+(\overline s)=X^-(s)$ this gives that $Z_{\overline s}$ is the connected component of $X\setminus T$ containing $X^-(s)$).
	\item For a vertex $v\in V$, the subset $X_v$ is defined as the (disjoint) union of the connected components $X_c$ of $X\setminus T$ such that $o(e_c)=v$ in the tree $X$, and which are not one of the components $Z_s$ described just above.
	\item Given $e\in E$ such that $o(e)=v$ is a vertex of $T$, we define the subset $X_v^e$ to be the union of the connected components $X_c$ of $X\setminus T$ such that $e_c$ is the image of $e$ by some element of the vertex group $G_v$. Observe that with this choice we have $X_v=\bigsqcup_{e\in \St_{\overline X}(v)}X_v^e$, where $\St_{\overline X}(v)=\left \{e\in E\mid o(e)=v\right \}$ denotes the star of $v$ in the quotient graph $\overline X$. 
	
	\begin{rem}\label{r:extra}
		Note that the definition of $X_v$ includes also components $X^s_v$ for complementary edges $s\in S$ with $o(s)=v$; moreover all subsets $X_v^e$ are pairwise disjoint. We do not require these particular conditions in Definition \ref{d:generalized_ping-pong}, but they will be important for the notion of ping-pong partitions on the circle that will appear as Definition \ref{d:markov-partition}.
	\end{rem}
\end{itemize}

It is routine to check that the family of subsets $\boxminus=\{X_v,Z_s\}$ defined above forms an interactive family. We give some detail of proof as it will important for Theorem \ref{t:main}.

\begin{prop}\label{p:arboreal_ping-pong}
	Given an orientation-preserving proper action $\alpha:G\to \Isom_+(X)$ of a group $G$ on a tree $X$, and a connected fundamental domain $T\subset X$, the family of subsets $\boxminus=\{X_v,Z_s\}$ defined above constitutes an interactive family for the action of $G$ on $X$.
\end{prop}

\begin{proof}
	As the subsets $X_v,Z_s$ are defined by taking union of different connected components, it is clear that they are pairwise disjoint. If there exists $e\in E$ such that $o(e)\in V$ and $G_{o(e)}\neq \alpha_e(A_e)$, then taking $g\in G_{o(e)}\setminus \alpha_e(A_e)$, one has $e.g\neq e$, $o(e.g)=o(e)$, so that $g$ sends $T$ inside the connected component $X_c$ of $X\setminus T$ containing $e.g$, and $X_c\subset X_{o(e)}$. This proves (IF~\ref{pp1}). The definition of the subsets $Z_s$ gives (IF~\ref{pp2}) easily. For (IF~\ref{pp3}), observe that $s$ is the edge of $X(s)$ sharing exactly one vertex with $T$ and oriented as $X(s)$, so that also $X^+(s)$ shares exactly one vertex with $s$; the subgroup $\alpha_s(A_s)$ fixes the edge $s$, and hence the vertex $X^+(s)\cap s$. This implies that $\alpha_s(A_s)$ preserves $Z_s$, so that (IF~\ref{pp3}) is proved. The other properties (IF~\ref{pp4}--\ref{pp8}) are deduced by analogous arguments.	
\end{proof}

\begin{rem}\label{r:strongIF4}
As an addendum to Remark \ref{r:extra},
we also observe that the collection of subsets $\{X_v^e,Z_s\}$ satisfies a strengthening of the properties of an interactive family. Indeed
the subsets $X_v^s$, for $s\in S$ such that $o(s)=v$, are such that $\left (G_v\setminus \alpha_s(A_s)\right )\left (Z_s\right )\subset X_v^s$, which is stronger than (IF \ref{pp4}), and additionally $\alpha_s(A_s)(X_v^s)=X_v^s$.
\end{rem}

\subsection{Finite-index subgroups}\label{sc:finite_index}

Let $H\subset G$ be a subgroup of finite index. Then the action $\alpha:G\to \Isom_+(X)$ restricts to an action of $H$. Thus we have subsets $W_c^H\subset H$.
Let $T^G$ be a connected fundamental domain for the action of $G$, and let $T^H\supset T^G$ be a connected fundamental domain for the action of $H$.
The inclusion $X\setminus T^H\subset X\setminus T^G$ induces a map $\iota:\pi_0(X\setminus T^H)\to \pi_0(X\setminus T^G)$, defined by $\iota(d)=c$ if and only if
$X^H_d\subset X_c^G$.

\begin{lem}\label{l:finite_index}
	With notation as above, for any connected component $c\in \pi_0(X\setminus T^G)$ we have
	\[
	W_c^G\cap H=\bigsqcup_{\iota(d)=c}W_d^H.
	\]
\end{lem}
\begin{proof}
	Clear from the definitions.
\end{proof}

\begin{ex}
	As an illustrative example, let us describe the situation for the classical action of $G=\PSL(2,\Z)\cong \Z_2*\Z_3$ on its Bass--Serre tree and the one of its derived subgroup.
	Let us denote by $a$ and $b$ the generators of the factors $\Z_2$ and $\Z_3$ respectively. Recall that the derived subgroup $H$, which is the kernel of the natural projection to $\Z_2\times \Z_3$ is a free subgroup of index $6$ in $G$.
	It is generated by the two commutators $[b,a]=bab^2a$ and $[b^2,a]=b^2aba$.
	The Bass--Serre tree $X$ of $G$ is formed by the cosets $\{\Z_2g,\Z_3g\}_{g\in G}$ as set of vertices, with edges given by pairs $(\Z_2g,\Z_3g)$. The tree is represented in Figure \ref{fig:BassSerrePSL}. The group $G$ acts on $X$ by right multiplication and a fundamental domain for the action is given by the edge $(\Z_2,\Z_3)$.
	One can choose as fundamental domain $T^H$ the subtree spanned by the three vertices $\Z_3,\Z_3[b,a],\Z_3[b^2,a]$, and removing any two of this vertices, as they belong to the same orbit (let us say that we keep $\Z_3$). The complement $X\setminus T^H$ has four connected components, two of them adjacent to the vertex $\Z_3$ and the other two adjacent to the vertices $\Z_3[b,a],\Z_3[b^2,a]$ respectively (and actually containing them). They refine the connected components of $X\setminus T^G$. See Figure~\ref{fig:BassSerrePSL2}.
	
	\begin{figure}[t]
		\[\includegraphics[scale=1]{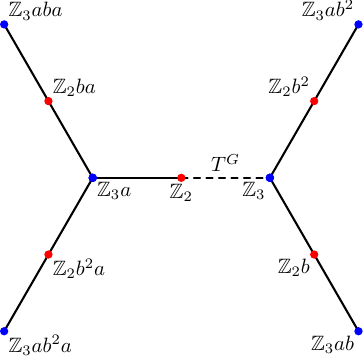}\]
		\caption{The Bass--Serre tree of $\PSL(2,\Z)$.}\label{fig:BassSerrePSL}
	\end{figure}
	
	\begin{figure}[t]
		\[\includegraphics[scale=1]{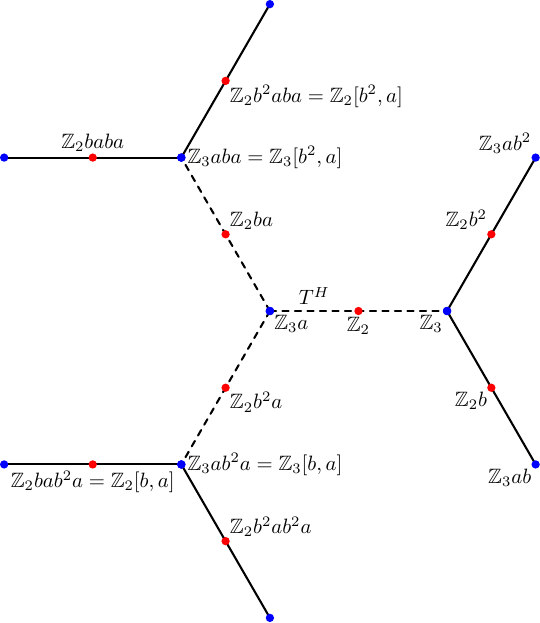}\]
		\caption{The action of the derived subgroup of $\PSL(2,\Z)$.}\label{fig:BassSerrePSL2}
	\end{figure}
\end{ex}

\section{Topological conjugacy}\label{s:topolconj}

In this section we formalize the notion of ping-pong partition for an action on the circle, and then we prove Theorem \ref{t:conj}, which states that the ping-pong partition determines the semi-conjugacy class of the action. The definition is similar to that appearing in \cite{BP} of \emph{basic partition}, even though the possible presence of torsion elements demands further requirements.

\subsection{Ping-pong partitions}\label{sc:definitions}

\begin{dfn}\label{d:gaps}
	Let  $\mathcal I$ be a collection of finitely many disjoint open intervals of the circle $\T$. A \emph{gap} of $\mathcal I$ is a connected component of the complement of $\bigcup_{I\in \mathcal I}I$ in $\T$. We denote by $\mathcal J$ the collection of gaps of the partition $\mathcal I$.
	
	Given a homeomorphism $g:\T\to\T$ and an interval $I\in \mathcal I$, we say that the image $g(I)$ is \emph{$\mathcal{I}$-Markovian} if it coincides with a union of intervals $I_0,\ldots,I_m\in \mathcal I$ and gaps $J_1,\ldots,J_m\in\mathcal J$ (where $m\ge 1$). We will also informally write that $g$ \emph{expands} the interval $I$.
\end{dfn}

\begin{rem}
	Unfortunately, the previous definition does not really capture the properties of the dynamics, as there is an ambiguity when $g(I)\in \cI$: sometimes we would like to consider it as a weak case of Markovian image, some other times we rather want to see it as a case of weak contraction. 
\end{rem}

For the next fundamental definition we recall our notation $\St_{\overline X}(v)=\left \{e\in E\mid o(e)=v\right \}$ for the star of a vertex $v\in V$ in the graph $\overline{X}=(V,E)$ (see \S\ref{sc:arboreal_ping-pong}).

\begin{dfn}[Ping-pong partition]\label{d:markov-partition}
	Let $G\subset \Homeo_+(\T)$ be a virtually free group and $(\alpha:G\to\Isom_+(X), T)$ a marking. Let $\cG=\{G_v,\alpha_s(A_s)s\}_{v\in V,s\in S}$ be the preferred system of generators for $G$.
	
	A \emph{ping-pong partition} for $(G,\alpha,T)$ is a collection $\Theta=\{U_v^e,V_s\}_{v\in V, e\in\St_{\overline X}(v), s\in S}$ of open subsets of the circle $\T$, satisfying the following properties.
	\begin{enumerate}[(PPP 1)]
		\item Letting $U_v=\bigcup_{e\in \St_{\overline X}(v)}U_v^e$, the family $\{U_v,V_s\}_{v\in V, s\in S}$ defines an interactive family in the sense of Definition~\ref{d:generalized_ping-pong}, with three additional requirements: for any $v\in V$
			\begin{itemize}
				\item the subsets $U_v^e$, for $e\in\St_T(v)$, are the subsets required for (IF \ref{pp5}),
				\item the subsets $U_v^s$, for $s\in S$ such that $o(s)=v$, are such that $\left (G_v\setminus \alpha_s(A_s)\right )\left (V_s\right )\subset U_v^s$, strengthening (IF \ref{pp4}), and moreover $\alpha_s(A_s)(U_v^s)=U_v^s$,
				\item the subsets $U_v^e$, for  $e\in\St_{\overline X}(v)$, are pairwise disjoint.
			\end{itemize}\label{ppp1}
		\item Every atom of $\Theta$ is the union of finitely many intervals.\label{ppp2}
		\item For every element $g\in \cG$ and every connected component $I$ of some $O\in\Theta$
		\begin{itemize}
			\item either there exists $O'\in\Theta$ such that $g(I)\subset O'$,
			\item or the image $g(I)$ is $\cI$-Markovian, where $\cI$ is the collection of connected components of elements of $\Theta$.
		\end{itemize}\label{ppp3}

	\end{enumerate}	
	In addition, if $\Theta$ defines a proper interactive family, we will say that  the ping-pong partition is \emph{proper}.
\end{dfn}

\begin{rem}\label{r:ppp5}
	Let $\mathcal J$ be the family of gaps of a ping-pong partition $\Theta$. Then it is not difficult to see that condition (PPP \ref{ppp3}) implies the following dual condition on gaps: for every $J\in\mathcal J$ and $g\in\cG$
	\begin{itemize}
		\item either there exists $O\in \Theta$ such that $g(J)\subset O$,
		\item or $g(J)\in \cJ$.
	\end{itemize}
\end{rem}

Many examples of ping-pong partitions are given in \cite[Section 3]{MarkovPartitions2}. There we give some basic constructions, for free and amalgamated products, and for HNN extensions, as well as some exotic examples of ping-pong partitions for locally discrete and minimal actions which are not of Fuchsian type. See also the brief discussion in Section \ref{s.examples}.

\subsection{Dynamics of a ping-pong partition}

\begin{prop}\label{p:MinimalSet}
	Let  $\Theta$ be a proper ping-pong partition for some marked virtually free group $(G,\alpha, T)$ of circle homeomorphisms. Then the action of $G$ on $\T$ has no finite orbit, and thus it admits a unique minimal invariant set $\Lambda$, which is either the circle or a Cantor set. Moreover, the minimal invariant set $\Lambda$ is contained in the closure of the intervals of the partition $\Theta$.
\end{prop}

\begin{proof}
	The fact that a minimal invariant set $\Lambda$ never is a finite orbit can be obtained by showing that the action of $G$ on $\T$ does not preserve a Borel probability measure, and this can be shown by using the inclusion relations of a proper interactive family.\footnote{Alternatively, one can pass to a finite-index subgroup $H\subset G$ which is free, and refine the partition $\Theta$ to obtain a ping-pong partition $\Theta_H$ for $H$. Then we get a ping-pong partition for a free group (in the classical sense), and the statement follows easily.}
\end{proof}

The following observation will be very useful in the rest of the section.

\begin{lem}\label{l:Markovian_image}
	Let $\Theta=\{U^e_v,V_s\}$ be a ping-pong partition for a marked virtually free group $(G,\alpha,T)$ of circle homeomorphisms, with collection of intervals $\cI$.
	\begin{enumerate}[1.]
		\item For a vertex $v\in V$ and edge $e\in\St_{\overline X}(v)$, let $I\subset U^e_v$ be a connected component of the partition, let $g\in\cG$ be a generator, and assume that $g(I)$ is $\cI$-Markovian. Then $g\in G_v$.
		\item For a complementary edge $s\in S$, let $I\subset V_s$ be a connected component of the partition, let $g\in\cG$ be a generator, and assume that $g(I)$ is $\cI$-Markovian. Then $g\in \alpha_{\overline s}(A_{\overline s})\overline s$.
	\end{enumerate}
\end{lem}

\begin{proof}
	We only detail the first part of the statement, the second one is analogous (using properties (IF \ref{pp2},\ref{pp4},\ref{pp8})).
	Assume $I\subset U^e_v$ for some vertex $v\in V$. Since $\Theta$ defines an interactive family by (PPP~\ref{ppp1}), then by (IF \ref{pp2},\ref{pp7}), an element which does not belong to the vertex group $G_v$ must send $I$ inside a connected component of some other $U_w$ or $V_s$, contradicting that $g$ expands $I$. So we must have $g\in G_v$.
\end{proof}

\begin{rem}\label{r.gen}
It follows from Lemma \ref{l:Markovian_image} that if two generators $g_1,g_2\in\cG$ and an interval $I\in\mathcal I$ are such that $g_1(I)$ and $g_2(I)$ are $\mathcal I$-Markovian then $g_1 g_2^{-1}$ is a generator, namely $g_1 g_2^{-1}\in\cG$.
\end{rem}

We deduce directly the following generalization of Remark~\ref{r:ppp5}.

\begin{lem}\label{l:images_gaps}
	With notation as in Lemma~\ref{l:Markovian_image}, let $\cJ$ denote the collection of gaps of $\cI$. Then for every generators $g_1,g_2\in \cG$ and gaps $J_1,J_2\in\cJ$, we have that the images $g_i(J_i)$ either coincide or are disjoint.
\end{lem}

\begin{proof}
	Suppose that the intersection  $g_1(J_1)\cap g_2(J_2)$ is not empty, but we do not have equality $g_1(J_1)=g_2(J_2)$. After Remark \ref{r:ppp5}, there exists $O\in \Theta$ (more precisely a connected component $I\subset O$), such that $g_1(J_1),g_2(J_2)\subset I\subset O$. Observe that both inclusions must be proper, as gaps are closed and the connected component $I$ is open. Thus, denoting by $I_i^-,I_i^+$ the connected components of atoms of $\Theta$ that are respectively left and right adjacent to the gap $J_i$, we must have by (PPP \ref{ppp3}) that also $g_i(I_i^\pm)\subset I$.
	We deduce that the inverse images $g_1^{-1}(I),g_2^{-1}(I)$ are $\cI$-Markovian, therefore by Remark \ref{r.gen} we must have $g_2^{-1}g_1=:h\in\cG$. We deduce that the image $h(J_1)$ intersects $J_2$, but does not equal $J_2$, and this contradicts Remark~\ref{r:ppp5}.
\end{proof}

The following lemma is a rephrasing of Lemma \ref{l:Markovian_image}, just passing to the inverses.

\begin{lem}\label{l:images_gaps_gen}
	With notation as in Lemma~\ref{l:images_gaps}, let $J\in\cJ$ be a gap and $g\in\cG$ a generator.
	\begin{enumerate}[1.]
		\item For a vertex $v\in V$ and edge$e\in\St_{\overline X}(v)$, assume $g(J)\subset U^e_v$. Then $g\in G_v$.
		\item For a complementary edge $s\in S$, assume $g(J)\subset V_s$. Then $g\in \alpha_{s}(A_{s})s$.
	\end{enumerate}
\end{lem}

\subsection{Refinement and its properties}\label{sc:refinement_ping-pong}

Let $\Theta=\{U^e_v,V_s\}$ be a ping-pong partition for a marked virtually free group $(G,\alpha,T)$ of circle homeomorphisms, with collection of intervals $\cI$ and gaps $\cJ$. We still denote by $\cG=\{G_v,\alpha_s(A_s)s\}$ the preferred generating system of $G$ and consider the collection $\widetilde{\cJ}:=\cG(\cJ)$ of images of gaps. We introduce the \emph{refined partition} $\wTh=\{\widetilde U^e_v,\widetilde V_s\}$ as follows: for any atom $O\in \Theta$, define $\wO:=O\setminus \bigcup_{\widetilde{J}\in\widetilde{\cJ}}\widetilde J$. In other words, 
$\wTh$ is the partition whose family of gaps is the collection of images $\widetilde{\cJ}$.
We have the following fundamental result.

\begin{prop}\label{p:refinement_is_ping-pong}
	With notation as above, the refined partition $\wTh=\{\widetilde U^e_v,\widetilde V_s\}$ is a ping-pong partition.
\end{prop}

\begin{proof}
	Let us start verifying (PPP~\ref{ppp1}): we have to check requirements (IF~\ref{pp1}--\ref{pp8}) and three additional properties.
	
	Observe that every new atom $\wO\in \wTh$ is obtained from the corresponding open subset $O\in \Theta$ by removing finitely many proper closed subsets, so if $O$ is nonempty, so is $\widetilde O$. This proves (IF~\ref{pp1},\ref{pp5}) and also the first and last extra requirements appearing in (PPP~\ref{ppp1}).
	
	We next consider the block of requirements (IF~\ref{pp2}--\ref{pp4}).
	For this, fix $s\in S$; observe that after Lemma~\ref{l:images_gaps_gen} and the relations in \eqref{eq:pres_fundamental_group}, we have
	\begin{equation}\label{eq:tVs}
	\widetilde V_s=V_s\setminus \alpha_s(A_s)s\left (\bigcup_{J\in\cJ}J\right )=V_s\setminus s\alpha_{\overline s}(A_{\overline s})\left (\bigcup_{J\in\cJ}J\right ).
	\end{equation}
	 Take $\wO\in\wTh\setminus \{\widetilde V_{\overline s}\}$, by (IF~\ref{pp2}) for $\Theta$ we have $s(\wO)\subset V_s$ and we want to verify that $s(\wO)$ intersects no image of gap in $V_s$. After \eqref{eq:tVs} such an image is of the form $sa(J)$ for some gap $J\in \cJ$ and element $a\in \alpha_{\overline s}(A_{\overline s})$; then we write
	 \[
	 s(\wO)\cap sa (J)=s\left (\wO\cap a(J)\right )
	 \]
	 and we observe that such intersection must be empty, since $a(J)\in \widetilde{\cJ}$ and because the atom $\wO$ is defined in such a way that it has empty intersection with every new gap in $\widetilde{\cJ}$. This proves (IF \ref{pp2}).
	 Property (IF \ref{pp3}) follows immediately from (IF \ref{pp3}) for $\Theta$ and \eqref{eq:tVs}. We are left to verify the strengthened version of (IF~\ref{pp4}) appearing in (PPP~\ref{ppp1}). After (PPP~\ref{ppp1}) for $\Theta$, we know that
	 $\left (G_{o(s)}\setminus \alpha_{s}(A_s)\right )(\widetilde V_s)\subset U_{o(s)}^s$ and, similarly as done for (IF \ref{pp2}), we want to verify that such images intersect no new gap $\widetilde J\in \widetilde{\cJ}$ in $U_{o(s)}^s$. For simplicity we write $v=o(s)$. After Lemma~\ref{l:images_gaps_gen} there exist an element $h\in G_v$ and a gap $J\in \cJ$ such that $h(J)=\widetilde J$. Take an element $g\in G_{v}\setminus \alpha_{s}(A_s)$ and write
	 \[
	 g(\widetilde V_s)\cap h(J)=g\left (\widetilde V_s\cap g^{-1}h(J)\right );
	 \]
	 observe that $g^{-1}h\in G_v$, thus $g^{-1}h(J)\in\widetilde{\cJ}$ and therefore the intersection $\widetilde V_s\cap g^{-1}h(J)$ is empty. This proves the strengthened version of (IF~\ref{pp4}).
	 Finally, for any vertex $v\in V$ and edge $e\in\St_{\overline X}(v)$, after Lemma \ref{l:images_gaps_gen} we have 
	 \begin{equation}\label{eq:tUev}
	 \widetilde U_v^e=U_v^e\setminus G_v\left (\bigcup_{J\in\cJ}J\right ).
	 \end{equation}
	 When $e\in S$ is a complementary edge, expression \eqref{eq:tUev} above and the second extra condition in (PPP~\ref{ppp1}) for $\Theta$ give the second extra condition in (PPP~\ref{ppp1}) for $\wTh$. When $e\in E_T$ is an edge of the spanning tree, the same argument proves (IF~\ref{pp6}) for $\wTh$.
	 Finally, the last two requirements (IF~\ref{pp7},\ref{pp8}) are also proved with arguments analogous to those for (IF~\ref{pp4}). This proves (PPP \ref{ppp1}).
	 
	 That $\wTh$ satisfies (PPP \ref{ppp2}) is obvious from the definition of $\wTh$. We must check (PPP \ref{ppp3}).
	 
	 In the following, we denote by $\Delta\subset \T$ the collection of endpoints of connected components $I\in\cI$ (equivalently of all $J\in \cJ$), and similarly $\widetilde{\Delta}\subset \T$ is the collection of endpoints for the refined partition $\wTh$. Note that as $\widetilde{\cJ}=\cG(\cJ)$ we have the analogous relation 
	 \begin{equation}\label{eq:refinement_Delta}
	 \widetilde{\Delta}=\cG(\Delta).
	 \end{equation}

\begin{claim1}
\label{l.consecutif}
Let $x,y\in\widetilde{\Delta}$ be the endpoints of an interval $\widetilde I$ of the refined partition $\widetilde \Theta$. Then there exists a generator $g\in\cG$ such that
$g(x)$ and $g(y)$ are both in $\Delta$, endpoints of (possibly distinct) intervals of $\Theta$.
\end{claim1}

\begin{proof}[Proof of Claim]
Without loss of generality we can assume that $x,y$ are respectively the left and right endpoints of $\widetilde I$. If $x,y\in\Delta$, we take $g=\mathrm{id}$ (the identity belongs to the generating set $\cG$ by definition). Now assume that only one of the two points belongs to $\Delta$, for example $x\in\Delta$ and $y\in\widetilde\Delta\setminus\Delta$. By definition there exists an interval $I\in\mathcal I$ that contains $\widetilde I\cup \{y\}\subset I$.  By equality \eqref{eq:refinement_Delta} there exists a generator $g\in\cG$ such that $g(y)\in\Delta$. Since $g(I)$ contains an element of $\Delta$ it must be $\mathcal I$-Markovian by (PPP \ref{ppp3}) for the partition $\Theta$. We deduce that $g(x)\in\Delta$. 

Now assume that both $x,y\in\widetilde\Delta\setminus\Delta$.
Under this assumption, there exists an interval $I\in\mathcal I$ that contains the closure of the interval $\widetilde I$.
Take generators $g_1,g_2\in\cG$ such that $g_1(x),g_2(y)\in\Delta$ and assume that $g_1(y)\notin\Delta$ and $g_2(x)\notin\Delta$.
Arguing as in the previous case, we have that both images $g_1(I)$ and $g_2(I)$ are $\mathcal I$-Markovian and thus by Remark \ref{r.gen}, the composition $g=g_2 g_1^{-1}$ belongs to $\cG$.
Moreover, since the interval $\widetilde{I}$ contains no point of $\widetilde{\Delta}=\mathcal{G}(\Delta)$, there exist two intervals $I_1,I_2\in\mathcal I$, with the following properties:
\begin{itemize}
	\item $g_1(x)$ is the leftmost point of $I_1$;
	\item $g_2(y)$ is the rightmost point of $I_2$;
	\item $g_1(y)\in I_1$ and $g_2(x)\in I_2$.
\end{itemize}
The first two properties give that the intersection $g_1^{-1}(I_1)\cap g_2^{-1}(I_2)$ equals $\widetilde I$, so $g(I_1)=g_2g_1^{-1}(I_1)$ has nontrivial intersection with $I_2$. However the third property implies that $g(I_1)=g_2g_1^{-1}(I_1)$ neither contains, nor is contained inside,   $I_2$. As $g\in\cG$ is a generator, this contradicts condition (PPP~\ref{ppp3}) for $\Theta$.
\end{proof}

\begin{claim1}
\label{l.refinment_Markov}
Let $x,y\in\widetilde\Delta$ be the endpoints of an interval $\widetilde I$ of $\widetilde\Theta$ and let $g\in\cG$ be  a generator. Assume that
$g(\widetilde I)\cap\widetilde\Delta\neq\emptyset$.
Then $g(x),g(y)\in\widetilde\Delta$.
\end{claim1}

\begin{proof}[Proof of Claim]
If $x,y\in\Delta$ then for every generator $g\in\mathcal G$ we have $g(x),g(y)\in\mathcal G(\Delta)=\widetilde{\Delta}$. So we will assume below that $x$ or $y$ is not in $\Delta$. Without loss of generality, we can assume $x\notin \Delta$.
Let $I\in\cI$ be the interval containing $\widetilde I$. Assume without loss of generality that $I\subset U^e_v$, the case $I\subset V_s$ being analogous.

By Claim \ref{l.consecutif} there exists a generator $h\in\cG$ such that $h(x),h(y)\in\Delta$. In particular the image $h(I)$ is $\mathcal I$-Markovian (because $x\in \widetilde{\Delta}\setminus \Delta$) and $h\in G_v$ by Lemma \ref{l:Markovian_image}.
On the other hand we have already proved the condition (PPP~\ref{ppp1}) that $\widetilde\Theta$ defines an interactive family  and, after the hypothesis $g(\widetilde I)\cap\widetilde\Delta\neq\emptyset$, the image $g(I)$ cannot be included inside a connected component of some other $U^f_w$ or $V_s$ (we use the same argument as in the proof of Lemma \ref{l:Markovian_image}). So $g\in G_v$ and we have $g'\in G_v\subset \cG$ where
$g'=gh^{-1}$.
So we have $g(x)=g'(h(x))\in g'(\Delta)\subset \widetilde\Delta$ and, similarly, $g(y)\in\widetilde\Delta$, thus proving the claim.
\end{proof}

We can now conclude. Let $\widetilde I\in\widetilde{\cI}$ be a connected component of the refined partition, and let $g\in\cG$ be a generator. Assume that $g(\widetilde I)$ is contained in no atom $\widetilde{O}\in\wTh$, we want to prove that the image $g(\widetilde I)$ is $\widetilde{\cI}$-Markovian. If $g(\widetilde I)$ contains a point of $\widetilde{\Delta}$, then we conclude by Claim~\ref{l.refinment_Markov}. It remains to exclude the case that $g(\widetilde I)$ is contained in a gap $\widetilde J\in\widetilde{\cJ}$, but this is done by repeating our previous arguments (or using (PPP~\ref{ppp1})). Indeed, assume $g(\widetilde I)\subset\widetilde J$; using (PPP \ref{ppp3}) for $\Theta$ we have that $\widetilde J$ is a ``new'' gap (that is, $\widetilde J\in\widetilde{\cJ}\setminus \cJ$), so there exists an interval $I\in \cI$ such that $\widetilde J\subset I$. Assume $I\subset U_v^e$ (the case $I\subset V_s$ is treated similarly), then by Lemma~\ref{l:Markovian_image} we have $g\in G_v$. Let $J\in \cJ$ and $h\in \cG$ be such that $h(\widetilde J)=J$. By Remark \ref{l:images_gaps_gen} we also have $h\in G_v$, so that the composition $hg\in G_v\subset \cG$ takes the interval $\widetilde I$ inside the gap $J$. But this contradicts the fact that the interval $\widetilde{I}$ has trivial intersection with every $\cG$-image of gaps in $\cJ$.
\end{proof}

\subsection{Equivalence of ping-pong partitions}

We keep the notation introduced in the previous section.

\begin{dfn}[Equivalence of partitions]\label{d:equivalence_PPP}
	Let $(G,\alpha,T)$ be a marked virtually free group, with preferred generating set $\cG=\{G_v,\alpha_s(A_s)s\}_{v\in V,s\in S}$ and let $\rho,\rho':G\to\Homeo_+(\T)$ be two representations having interactive families $\Theta=\{U^e_v,V_s\}$, $\Theta'=\{{U^e_v}',V'_s\}$ respectively. Denote by $\cI,\cI'$ the corresponding sets of connected components. A map $\theta:\cI\to\cI'$ is a \emph{ping-pong equivalence} if the following conditions are satisfied:
	\begin{enumerate}[(PPE 1)]
		\item the map $\theta$ is a bijection which preserves the cyclic ordering of the intervals;\label{ppe1}
		\item the map $\theta$ preserves the inclusions relations of the two ping-pong partitions:
		\begin{itemize}
			\item if a generator $g\in \cG$ and intervals $I_1, I_2\in \cI$ are such that $\rho(g)(I_1)\subset I_2$, then $\rho'(g)(\theta(I_1))\subset \theta(I_2)$, and, in case of equality, equality is preserved,
			\item whereas if a generator $g\in \cG$ and an interval $I$ are such that the image $\rho(g)(I)$ is $\mathcal I$-Markovian, union of intervals $I_0,\ldots,I_m\in\cI$ and gaps $J_1,\ldots, J_m\in\cJ$, then the image $\rho'(g)(\theta(I))$ is $\cI'$-Markovian, union of the intervals $\theta(I_0),\ldots,\theta(I_m)\in\cI'$ and the gaps $\theta(J_1),\ldots, \theta(J_m)\in\cJ$.
		\end{itemize} \label{ppe3}
	\end{enumerate}
\end{dfn}

\begin{rem}\label{r:multivalued}
	A ping-pong equivalence $\theta:\cI\to \cI'$ induces a ``map'' from $\Delta$ to $\Delta'$ (the collections of endpoints of $\cI$ and $\cI$ respectively), which we still denote by $\theta$; however such a map $\theta:\Delta\to\Delta'$ may fail to be a bijection, and even worse, it can happen that a point $x\in \Delta$ has two different images in $\Delta'$. This is because we do not want to distinguish the cases where two intervals in $\cI$ are separated by a single point or a nontrivial gap (the semi-conjugacy class of the action is the same).
\end{rem}

\begin{rem}[Equivariance]\label{rem:equivariance}
	With abuse of notation, let $\theta:\Delta\to \Delta'$ denote the (possibly multivalued) map induced by the ping-pong equivalence $\theta:\cI\to \cI'$, as discussed in the previous Remark~\ref{r:multivalued}.
	As a consequence of the definition of ping-pong equivalence, if a generator $g\in\cG$ and points $x_1,x_2\in \Delta$ are such that $\rho(g)(x_1)=x_2$ then $\rho'(g)(\theta(x_1))=\theta(x_2)$.
\end{rem}


\subsection{Proof of Theorem~\ref{t:conj}}\label{sc:proof_thmB}

\begin{lem}\label{lem:extension_refinement}
With notation as above, let $\Theta$ and $\Theta'$ be two ping-pong partitions with a ping-pong equivalence $\theta:\cI\to\cI'$. Then the map
	$\theta$ extends to  a ping-pong equivalence $\widetilde \theta$, where $\widetilde \cI$ and $\widetilde \cI'$ denote the collections of connected components of $\widetilde \Theta$ and $\widetilde \Theta'$, respectively.
\end{lem}

\begin{proof}
	We will work with the (possibly multivalued) induced map $\theta:\Delta\to\Delta'$  (Remark \ref{r:multivalued}).
	Let $y\in\widetilde\Delta\setminus \Delta$ and consider a generator $g\in \cG$ and a point $x\in \Delta$ such that $\rho(g)(x)=y$. Equivariance (Remark \ref{rem:equivariance}) forces to define
	\[
	\widetilde\theta(y)=\widetilde\theta(\rho(g)(x)):=\rho'(g)(\theta(x)).
	\]
	We have to prove that the previous definition is coherent with the dynamics. In order to do this, consider  the interval $I\in\mathcal I$ of the partition $\Theta$ which contains the point $y$. Assume that we have two pairs $(g_1,x_1)$ and $(g_2,x_2)$ satisfying $g_i\in\cG$ and $x_i\in\Delta$ so that $\rho(g_1)(x_1)=\rho(g_2)(x_2)=y$. Then the images $\rho(g_1)^{-1}(I)$ and $\rho(g_2)^{-1}(I)$ must be $\cI$-Markovian. By Remark \ref{r.gen} we must have $g_2^{-1}g_1\in\cG$. 
	Applying (PPE \ref{ppe3}) to the intervals $I_1$ and $I_2$ having $x_1$ and $x_2$ respectively as an endpoint, we see that
	$\rho'(g_2^{-1}g_1)(\theta(x_1))=\theta(x_2)$,
	which proves that the image $\widetilde\theta(y)$ is well defined. Note that this argument shows directly that $\widetilde\theta$ satisfies the two conditions of (PPE \ref{ppe3}) (for this, note that if one has an inclusion $\rho(g)(I_1)\subset I_2$, then the image $\rho(g^{-1})(I_2)$ is $\cI$-Markovian, so the inclusions relations are prescribed by the images of the endpoints).
	
	Let us prove (PPE \ref{ppe1}) for $\widetilde\theta$. Since $\widetilde\theta=\theta$ in restriction to the set of endpoints $\Delta$, we see that $\widetilde\theta$ preserves the cyclic ordering of $\Delta$. Hence it is enough to check that for every interval $I\in\mathcal I$ the restriction $\widetilde\theta\restriction_{I}$ preserves the cyclic order. We assume that $I\cap\widetilde\Delta\neq\emptyset$. Note that by construction of $\widetilde\theta$ we have the equivariance relation which holds for every generator $g\in\cG$ expanding $I$:
	\[\widetilde\theta(I\cap\widetilde\Delta)=\rho'(g)^{-1}\left (\theta\left(\rho(g) (I)\cap\Delta\right)\right ).\]
	Since $\rho(g)$ and $\rho'(g)$ preserve the orientation and since (PPE \ref{ppe1}) holds for $\theta$, we see that $\widetilde\theta$ preserves the cyclic ordering of $I\cap\widetilde\Delta$. This achieves the proof of (PPE \ref{ppe1}) for $\widetilde\theta$.
\end{proof}

\begin{proof}[Proof of Theorem~\ref{t:conj}]
	Let $\rho,\rho':(G,\alpha,T)\to \Homeo_+(\T)$ be two representations of a marked virtually free group, with equivalent proper ping-pong partitions $\Theta$ and $\Theta'$ respectively.
	Write $\Delta_0=\Delta$ and recursively $\Delta_{k+1}=\widetilde \Delta_k=\rho(\cG)(\Delta_{k})$. Similar sets are defined for the representation $\rho'$.
	Then by Proposition \ref{p:refinement_is_ping-pong} and Lemma \ref{lem:extension_refinement}, one gets ping-pong equivalences $\theta_k:\Delta_k\to\Delta_k'$.
	Observe that the sets $\Delta_k$ define an increasing sequence of finite sets, with the union $\Delta_{\infty}=\bigcup_{k\in\N}\Delta_k$ accumulating on the minimal invariant set $\Lambda$ (because $\Delta_{\infty}=\bigcup_{k\in\N}\Delta_k$ is the union of the orbits of points in $\Delta_0$). Observe that as the ping-pong partitions are proper, by Proposition~\ref{p:MinimalSet} we have that the minimal invariant set  for $\rho(G)$ is unique, and either the circle or a Cantor set. Similar considerations work for $\Delta'_\infty$.
	
	By a limit process the maps $\theta_k:\Delta_k\to\Delta_k'$ define a circular order bijection $\theta_\infty:\Delta_\infty\to\Delta'_\infty$ which extends uniquely to a continuous equivariant bijection $\theta_\infty:\Lambda\to\Lambda'$ which preserves the circular order.
	By the definition of ping-pong equivalence we get that $\rho'(s)(\theta_\infty(x))=\theta_\infty(\rho(x))$ for all points $x\in\Lambda$. The lift of $\theta_\infty$ to $\Lambda+\Z\subset \R$ extends to a monotone non-decreasing map defined on $\R$, which  commutes with integer translations, and this gives the semi-conjugacy.
\end{proof}

\section{DKN partitions for virtually free groups}\label{s:topskel_gen}

In this section we use the arboreal partitions described in \S\S\ref{sc:arboreal_part}--\ref{sc:finite_index} to obtain a ping-pong partition of the circle for a locally discrete, virtually free group~$G\subset \Diff_+^\omega(\T)$ as done by Deroin--Kleptsyn--Navas in \cite{DKN2014} for free groups.

\subsection{From the arboreal partition to the DKN partition}

Suppose that $G\subset \Diff^\omega_+(\T)$ is a  locally discrete, virtually free group (we are supposing that $G$ is not virtually cyclic).
Fix a word norm $\|\cdot\|$ on the group $G$ given by a finite generating system. 
Let $\alpha:G\to\Isom_+(X)$ be a proper cocompact action on a locally finite tree, with fundamental domain $T$.
Then, as we described in \S \ref{ssc.dyn_arb}, we have an arboreal partition of $G$, defined by the subsets $W_c$.
Inspired by \cite{DKN2014}, we ``push'' the arboreal partition to the circle, defining the following subsets:
\[
U_{E}:=\left\{x\in\T\,\middle\vert\,\exists\text{ neighborhood }I_x\ni x\text{ s.t.~}\lim_{n\to\infty}\sup_{g\notin E,\|g\|\ge n}|g(I_x)|=0\right\}
\]
for any subset  $E$ of $G$.
We shall use frequently the shorthand notation $U_c$ instead of $U_{W_c}$. The sets $U_c$ determine all the sets of the form $U_{W_ch}$:
\begin{lem}\label{l:top-alg}
Given an element $h\in G$ and a subset $E\subset G$ in the group, the subsets $U_{Eh}$ and $h^{-1}(U_{E})$ of the circle coincide.
\end{lem}
\begin{proof}
Let us first make a change of variable
\begin{equation}
\label{eq:limsup}
\sup_{g\notin Eh,\|g\|\ge n}|g(I_x)|=\sup_{gh^{-1}\notin E,\|g\|\ge n}|g(I_x)|=\sup_{g\notin E,\|{gh}\|\ge n}|gh(I_x)|.
\end{equation}
Then by the triangle inequality, the quantity in \eqref{eq:limsup} can be bounded as follows: 
\[\sup_{g\notin E,\|{g}\|\ge n+\|h\|}|gh(I_x)|\le  \sup_{g\notin E,\|gh\|\ge n}|gh(I_x)|\le \sup_{g\notin E,\|g\|\ge n-\|h\|}|gh(I_x)|.\]

As the subset $h(I_x)$ is a neighborhood of the point $h(x)$, taking the limit on each term as $n\to \infty$, gives that if $x\in U_{Eh}$ then $h(x)\in U_{E}$, as wanted.
\end{proof}
Furthermore, the inclusion relations for the sets $W_cg$ can be naturally transferred to the sets $U_{W_cg}$ (this should  be compared with \cite[\S5.1]{BP}).
\begin{lem}\label{l:inclusionU}
For any pair of subsets $E\subset F$ of the group $G$, one has
 $U_{E}\subset U_{F}$.
\end{lem}

\begin{proof}
As we have $E\subset F$, we have the reversed inclusion of the sets
$
\{g\notin E\}\supset \{g\notin F\}.
$
This implies that for any subset $I\subset \T$, for any integer $n\in\N$ one has the inequality
\begin{equation}\label{eq:comp_sup}
\sup_{g\notin E,\|g\|\ge n}|g(I)|\ge \sup_{g\notin F,\|g\|\ge n}|g(I)|.
\end{equation}
Take any point $x\in U_{E}$. By definition, there exists an open neighborhood $I_x\ni x$ such that
\[\lim_{n\to\infty}\sup_{g\notin E,\|g\|\ge n}|g(I_x)|=0.\]
As \eqref{eq:comp_sup} holds for every integer $n\in\N$, we must have inequality also when taking the limit as $n\to \infty$, thus
\[\lim_{n\to\infty}\sup_{g\notin F,\|g\|\ge n}|g(I_x)|=0.\]
Hence the inclusion $U_{E}\subset U_{F}$.
\end{proof}

Our aim is to prove that the collection of subsets $\{U_c^G\}_{c\in \pi_0(X\setminus T^G)}$ defines a DKN partition of the circle, in the sense that it possesses properties which are analogue to those of the partition from Theorem~\ref{t:DKNfree} for free groups.

\begin{thm}\label{t:DKNmarkov3}
Let $G\subset\Diff_+^\omega(\T)$ be a locally discrete, virtually free group, with minimal invariant set $\Lambda$.
Let $\alpha:G\to \Isom_+(X)$ be a proper action on a locally finite tree and let $T$ be a connected fundamental domain.
The collection of subsets $\{U_c\}_{c\in \pi_0(X\setminus T)}$ satisfies the following properties:
\begin{enumerate}[1.]
\item every subset $U_c$  is open;

\item every subset $U_c$ has finitely many connected components;\label{i:finite_union_DKN}

\item any two different subsets $U_c$ have empty intersection inside $\Lambda$;

\item the union of the subsets $U_c$ covers all but finitely many points of $\Lambda$;

\item for any element $g\in W_{c\rightarrow d}$, the inclusion  
 $g(U_c)\subset U_d$ holds.\label{i:inclusion_relation}
\end{enumerate}
\end{thm}

\begin{rem}\label{r.groupe_darete}
The condition $g\in W_{c\to d}$ for Item \ref{i:inclusion_relation}.~is detailed in \S \ref{ssc.dyn_arb}, and in particular it is characterized by Proposition \ref{p:inclusionW}.	Here we simply remark the following.	
The marking $(\alpha,T)$ of the group $G$ gives an isomorphism $G\cong \pi_1(\overline X;G_v,A_e)$.
Let the element $g\in G_v$ and  $c,d\in\pi_0(X\setminus T)$ be such that the connected components $X_c$ and $X_d$ are adjacent to the vertex $v$ (when lifting $v\in V$ to a vertex of $T\subset X$). In this case, the condition $g\in W_{c\to d}$ implies $g^{-1}\in W_{d\to c}$ (Lemma \ref{l:arboreal_vertex}). As a consequence, Item \ref{i:inclusion_relation}.~of Theorem \ref{t:DKNmarkov3} above gives the equality $g (U_c)=U_d$.
(Observe that we can possibly have $c=d$, and in such case $g\in\alpha_e(A_e)$ for some edge $e\in E$ satisfying $o(e)=v$.)
\end{rem}

In what remains, we proceed step by step, as in \cite{DKN2014}, separating the proofs of the different properties of the sets $U_c$. We start in \S\ref{ssc:basic} with the most elementary proofs. In Theorem~\ref{t:DKNfree} the second and fourth properties were the most delicate part to verify, requiring sophisticated tools of one-dimensional dynamics. Here we can use the good behaviour of the partition when passing to a finite-index free subgroup to reduce to the already proved Theorem~\ref{t:DKNfree}. This is done in \S\ref{ssc:finite}.

\subsection{Elementary properties}\label{ssc:basic}

\begin{lem}
The sets $U_c$ are open.
\end{lem}

\begin{proof}
This follows from the definition (in \cite[Proposition 2.4]{DKN2014}, where the definition of the sets $U_c$ is different, this property follows from a classical argument of control of affine distortion).
\end{proof}

\begin{lem}\label{l:empty_int}
For two different connected components $c,d\in\pi_0(X\setminus T)$, the intersection $U_c\cap U_d$ is contained in the complement of the minimal set $\Lambda$ of $G$.
\end{lem}

\begin{proof}
Indeed, for a point $x$ that lies in the intersection $U_c\cap U_d$, we have that there exists a neighborhood $I_x\ni x$ such that 
\[\lim_{n\to\infty}\sup_{g\notin W_{c}\cap W_d,\|g\|\ge n}|g(I_x)|=0.\]
However, the intersection $W_c\cap W_d$ is empty, so the neighborhood $I_x$ is contracted by every sufficiently long iteration of elements in $G$.
This is in contradiction with the existence of elements with hyperbolic fixed points in the minimal set (Sacksteder's theorem \cite[\S3.2]{Navas2011}). See \cite[Proposition 2.5]{DKN2014}.
\end{proof}

\begin{lem}\label{l:ping-pong}
Given an element $g\in W_{c\rightarrow d}$, then the inclusion $g(U_c)\subset U_{d}$ holds.
\end{lem}

\begin{proof}
Indeed, by definition of $W_{c\to d}$, we have $W_cg^{-1}\subset W_d$, and using Lemmas~\ref{l:top-alg} and \ref{l:inclusionU}, we get $g(U_c)=U_{W_cg^{-1}}\subset U_d$.
\end{proof}

\subsection{A finite number of connected components}\label{ssc:finite}

In this subsection, we will make use of the notation introduced in \S\ref{sc:finite_index}.

\begin{prop}\label{p:finite}
Each set $U_c$ has finitely many connected components.
\end{prop}

This follows from the proposition below together with property \ref{i:finite_free}.~of Theorem~\ref{t:DKNfree}.

\begin{prop}\label{p:freesubgroup}
With notation as above, let $H\subset G$ be a free subgroup of finite index. For any connected component $c\in \pi_0(X\setminus T^G)$ we have the equality
\[
U_c^G\cap \Lambda=\left\{(\text{finitely many points})\sqcup\bigsqcup_{\iota(d)=c}U_d^H\right\}\cap \Lambda.
\]
In particular, the union of the subsets $U_c^G$ covers all but finitely many points of $\Lambda$.
\end{prop}

Before giving the proof, we start with some preliminary lemmas.

\begin{lem}\label{l:quasi-isomU}
With notation as above, let $H\subset G$ be a  free subgroup of finite index. Then for any connected component $c\in \pi_0(X\setminus T)$, we have the equality
\[
U^G_c=\left \{
x\in\T\,\middle\vert\,\exists\text{ neighborhood }I_x\ni x\text{ s.t.~}\lim_{n\to\infty}\sup\left\{|g(I_x)|\,;\, g\in H,g\notin W_c^G,\|g\|\ge n\right\}=0
\right \}.
\]
\end{lem}
\begin{proof}
In other words, we want to prove that in order to define the set $U^G_c$ it is enough to take account of the contraction of a neighborhood under the finite-index subgroup $H$.
For this, let $\Sigma=\{t_1,\ldots,t_k\}$ be a set of representatives of $G/H$, that is, any element $g\in G$ may be written in the (unique) form
\[
g=t h,\quad \text{with }h\in H \text{ and }t\in\Sigma.
\]
Let $D$ be a bounded subset in $X$ containing all the images $T.\Sigma$. Then there exists a large integer $n_0\in\N$ (depending on the subset $D$), such that when an element $h\in H$ satisfies $\|h\|\ge n_0$, then the images $D.h$ and $T.h$ are in the same connected component of $X\setminus T$.
For any integer $n\ge n_0$ one has that the following subset of $G$,
\[
\left\{g=th\,\middle\vert\, t\in \Sigma, h\in H, D.h\nsubseteq X^G_c,\|h\|\ge n\right\}\]
coincides with
\[
\left\{g\in G\,\middle\vert\, g\notin W_c^G,\|g\|\ge n\right\},
\]
up to finite subset. Hence the subset
\[
\left\{g=th\,\middle\vert\, t\in \Sigma, h\in H, T.h\nsubseteq X^G_c,\|h\|\ge n\right\}
\]
coincides with
\[
\left\{g\in G\,\middle\vert\, g\notin W_c^G,\|g\|\ge n\right\}
\]
up to  finite subset.

By compactness, for any $\eps>0$ there exists $\delta>0$ such that if a subset $J\subset \T$ verifies $|J|<\delta$, then $|t(J)|<\eps$ for any element $t\in\Sigma$.
For any $\delta>0$, there exists some large integer $n\ge n_0$ such that if an element $h\in H$ verifies $T.h\nsubseteq X^G_c,\|h\|\ge n$ and $I\subset \T$ satisfies $|h(I)|<\delta$, then
\[
|g(I)|=|th(I)|\le \eps. 
\]
This gives the statement.
\end{proof}

Recall that we are denoting by $\iota:\pi_0(X\setminus T^H)\to \pi_0(X\setminus T^G)$ the map induced by the inclusion $T^H\supset T^G$.

\begin{lem}\label{l:inclusionUH}
With the previous notation, if $d\in \pi_0(X\setminus T^H)$ and $c\in\pi_0(X\setminus T^G)$ satisfy $\iota(d)=c$, then we have the inclusion
\[
U^H_d\subset U^G_c.
\]
\end{lem}
\begin{proof}By Lemma \ref{l:finite_index},
we have $W^H_d\subset W^G_c\cap H$. Therefore after Lemma \ref{l:inclusionU} we get
\[
U^H_d\subset U^H_{W^G_c\cap H}.
\]
On the other hand the previous Lemma~\ref{l:quasi-isomU} says that the set $U^H_{W^G_c\cap H}$ coincides with $U^G_c$.
\end{proof}

\begin{lem}\label{l:minimal_subgroup}
Let $G$ be a virtually free group acting on the circle with a minimal invariant set $\Lambda$ which is not a finite set. Let $H\subset G$ be a finite-index subgroup. Then $\Lambda$ is the minimal invariant set for the action of $H$.
\end{lem}

\begin{proof}
By the so-called Poincar\'e lemma, if the group $G$ contains a subgroup $H$ of finite index, then it contains a normal subgroup $H'$ of finite index, which moreover is contained in $H$ (one defines $H'$ to be the intersection of the conjugates of $H$).
Let $\Lambda'$ be a minimal invariant set for $H'$. As $H'$ is normal in $G$, all the images $g(\Lambda')$ are invariant sets for $H'$. As the subgroup $H'$ has finite index in $G$, there are at most finitely many different images $g_1(\Lambda'),\ldots,g_n(\Lambda')$. Moreover, they must verify
\[
\Lambda=g_1(\Lambda')\cup \cdots \cup g_n(\Lambda').
\]
Since $\Lambda$ is infinite, so must be $\Lambda'$ and therefore $\Lambda'$ is the unique minimal invariant set for $H'$. Therefore $G$ preserves $\Lambda'$ and so $\Lambda'=\Lambda$, as wanted.
\end{proof}

\begin{proof}[Proof of Proposition~\ref{p:freesubgroup}]
From Lemma~\ref{l:inclusionUH} we get the inclusion
\[
U_c^G\supset\bigsqcup_{\iota(d)=c}U_d^H.
\]
On the other hand Theorem~\ref{t:DKNfree} applied to the subgroup $H$ tells that the partition $U_d^H$ covers all the minimal set $\Lambda$ but finitely many points (after Lemma~\ref{l:minimal_subgroup}, the minimal set for $H$ is the same as for $G$).
\end{proof}
This clearly implies Proposition~\ref{p:finite} and also concludes the proof of Theorem~\ref{t:DKNmarkov3}.

\subsection{Making the DKN partition a ping-pong partition}

After having established Theorem~\ref{t:DKNmarkov3}, we want to prove that the partition into subsets $\{U_c\}_{c\in\pi_0(X\setminus T)}$ gives an interactive family in the sense of Definition~\ref{d:generalized_ping-pong}. This is an immediate consequence of the fact that the DKN partition comes from the arboreal partition, preserving the inclusion relations (Lemmas \ref{l:top-alg} and \ref{l:inclusionU}), and we have already remarked that the latter defines a ping-pong partition (Proposition~\ref{p:arboreal_ping-pong}). However, some extra verification is needed to prove that the interactive family is proper: condition (IF~\ref{pp10}) does not follow directly from Proposition~\ref{p:arboreal_ping-pong}, but is a consequence of the dynamics of the action (Theorem~\ref{t:DKNmarkov3}).

By Theorem~\ref{t:DKNmarkov3}, the open subsets $U_c$ are individually the union of finitely many open intervals, their union covers the minimal set $\Lambda$ except a finite number of points, and two of them can intersect only in the complement of $\Lambda$. The fact that two subsets $U_c$ may intersect gives problem to properly define a ping-pong partition. We contour this by \emph{reducing} the subsets $U_c$ similarly as done after Theorem \ref{t:DKNfree}: if an endpoint of a connected component of $U_c$ belongs to a \emph{gap} $J\subset \T\setminus \Lambda$ (that is, a connected component of the complement $\T\setminus \Lambda$), we remove $U_c\cap \overline J$ from $U_c$. As $U_c$ has finitely many connected components, we only have to do this operation finitely many times. We denote by $\hat U_c\subset U_c$ the open subset obtained at the end of this removal process. Using the invariance of the set $\Lambda$, it is not difficult to verify that the new collection $\hat U_c$ also verifies the list of properties of Theorem~\ref{t:DKNmarkov3}, with the additional property that all the subsets $\hat U_c$ are now pairwise disjoint, and the complement of their union is a non-empty finite union of gaps and finitely many points of $\Lambda$. (This reduction leads to a partition which is closer to those appearing in \cite{BP}.)

Next, we gather together different subsets $\hat U_c$ to obtain an interactive family, following the procedure described in \S\ref{sc:arboreal_ping-pong} for the arboreal partition. This gives a family of open subsets $\Theta_{DKN}=\{U^e_v,V_s\}_{v\in V,s\in S}$ of the circle that we call the \emph{DKN ping-pong partition} 
We denote by $\cI$ the collection of connected components of elements of $\Theta_{DKN}$ and by $\cJ$ the collection of gaps of $\cI$ (see Definition \ref{d:gaps}).

Let us verify that $\Theta_{DKN}$ is indeed a ping-pong partition, in the sense of Definition \ref{d:markov-partition}.

\begin{lem}\label{l:DKN_IF}
	Let $G\subset \Diff^\omega_+(\T)$ be a locally discrete, virtually free group of real-analytic circle diffeomorphisms. For any marking $(\alpha:G\to \Isom_+(X),T)$, the DKN  ping-pong partition $\Theta_{DKN}=\{U^e_v,V_s\}_{v\in V,s\in S}$ verifies \emph{(PPP~\ref{ppp1})}.
\end{lem}

\begin{proof}
Properties (IF \ref{pp1},\ref{pp5}) are satisfied by the DKN ping-pong partition,  by construction. After \ref{i:inclusion_relation}.~in Theorem~\ref{t:DKNmarkov3}, and Lemmas \ref{l:top-alg} and \ref{l:inclusionU}, the dynamics on the subsets $U_c$ behaves exactly as that on the connected components $X_c$ of $X\setminus T$ (where $X$ is the tree of the marking). We have already proved (Proposition \ref{p:arboreal_ping-pong}) that the arboreal ping-pong partition $\{X_v,Z_s\}$ defines an interactive family, and thus we obtain directly from \ref{i:inclusion_relation}.~in Theorem~\ref{t:DKNmarkov3}, Lemmas \ref{l:top-alg} and \ref{l:inclusionU} that the corresponding collection defined with the subsets $U_c$ satisfies the inclusion relations (IF \ref{pp2}--\ref{pp4},\ref{pp6}--\ref{pp8}) (but remember that they are not necessarily disjoint). Next, we observe that the subsets $\hat U_c$ defining the $U^e_v,V_s$ are obtained from the subsets $U_c$ by cutting them at well-chosen points of the minimal invariant subset $\Lambda$, which is $G$-invariant. This implies that the same inclusion relations  (IF \ref{pp2}--\ref{pp4},\ref{pp6}--\ref{pp8}) that are valid for the collections of $U_c$ are also valid for the collections $U^e_v,V_s$ of $\hat U_c$. The extra three conditions appearing in (PPP~\ref{ppp1}) are also satisfied by construction of $\Theta_{DKN}$, as they are valid for the the arboreal partition (Remarks~\ref{r:extra} and \ref{r:strongIF4}).
\end{proof}

That every element of $\Theta_{DKN}$ is the union of finitely many intervals is a direct consequence of Item~\ref{i:finite_union_DKN}.~of Theorem~\ref{t:DKNmarkov3}. So (PPP~\ref{ppp2}) is also verified. 

\begin{lem}\label{l:DKN_PPP3}
	Under the hypotheses of Lemma \ref{l:DKN_IF}, the partition $\Theta_{DKN}$ satisfies condition \emph{(PPP~\ref{ppp3})} of Definition \ref{d:markov-partition}.
\end{lem}

\begin{proof}
	By the properties of an interactive family (that we have verified in Lemma \ref{l:DKN_IF}), we are reduced to consider the two cases for which it is not clear that a contraction occurs.
	
	\begin{claim}
		If $I$ is a connected component of $U^e_v$  and $g\in G_v$ then either $g(I)$ is contained in an interval $I'\in\cI$ or $g(I)$ is $\cI$-Markovian.
		
		Similarly, if $I$ is a connected component of $V_s$ and $g\in \alpha_{\overline s}(A_{\overline s})\overline s$, then either $g(I)$ is contained in an interval $I'\in\cI$ or $g(I)$ is  $\cI$-Markovian.
	\end{claim}
	
	\begin{proof}[Proof of Claim]
		We only treat the first situation, the second one is completely analogous.
		We want to prove that if the image $g(I)$ contains points of a gap, then it must be $\cI$-Markovian. This follows from the definition of the DKN partition $\Theta_{DKN}$: as gaps in $\cJ$ are points or the closure of connected components of the complement of $\Lambda$, the image $g(I)=(x_-,x_+)$ cannot satisfy that a right (respectively, left) neighborhood of $x_-$ (respectively, $x_+$) is contained in a gap, otherwise $\Lambda$ would have isolated points.
	\end{proof}
	This gives (PPP \ref{ppp3}) for the partition $\Theta_{DKN}$.
\end{proof}

Summarizing, we have:

\begin{prop}\label{p:DKN_is_pingpong}
	Let $G\subset \Diff^\omega_+(\T)$ be a locally discrete, virtually free group of real-analytic circle diffeomorphisms. For any marking $(\alpha:G\to \Isom_+(X),T)$, the DKN ping-pong partition is a ping-pong partition (in the sense of Definition~\ref{d:markov-partition}).
\end{prop}

\subsection{Getting a proper ping-pong partition}

We were not able to find an elementary argument to show that the DKN partition defines a \emph{proper} ping-pong partition (and we are not even sure that this is always the case!). Anyhow, refining the DKN ping-pong partition as in \S\ref{sc:refinement_ping-pong} gives what we need.

\begin{prop}\label{p:DKN_is_proper}
	Let $G\subset \Diff^\omega_+(\T)$ be a locally discrete, virtually free group of real-analytic circle diffeomorphisms. For any marking $(\alpha:G\to \Isom_+(X),T)$, there exists some $k\in\N$ such that after refining $\Theta_{DKN}$ $k$ times, one gets a proper interactive family.
\end{prop}

\begin{proof}
By Proposition \ref{p:refinement_is_ping-pong} and Lemma~\ref{l:DKN_IF}, every refinement of the DKN ping-pong partition is an interactive family.
For $k\ge 1$ we denote by $\Theta_{k}$ the $k$-th refinement of $\Theta_{DKN}$.
The verification of (IF~\ref{pp9}) that the edge groups act faithfully is an easy consequence of the fact that they are finite cyclic groups acting on the circle (hence every nontrivial element has no fixed point).	
Observe also that the circumstance under which (IF~\ref{pp11}) must be verified never occurs under our assumptions, as this would imply that $G$ is virtually cyclic (Remark \ref{pipipi}).

We want to prove that there exists $k\in \N$ such that $\Theta_k$ satisfies (IF~\ref{pp10}).
If $U_v=\emptyset$ for every vertex $v\in V$, then there is nothing to check. So we can assume $U_v\neq \emptyset$ for some vertex $v\in V$.
Verifying (IF~\ref{pp10}) for $\Theta_{DKN}$ comes up to verify that there exists a gap $J\in\cJ$ which is sent inside a connected component $I\in\cI$ of an open set $U_v$, for some $v\in V$, by some generator $g\in G_v$.
Observe first that a gap $J\in \cJ$ must be sent inside some interval $I\in\cI$ by some generator $g\in\cG$, otherwise the collection of gaps $\cJ$ would be $G$-invariant, giving a finite orbit for the action of $G$ (recall that this does not occur because $G$ is locally discrete and not virtually cyclic).
What may happen, invalidating (IF~\ref{pp10}),  is that for every generator $g\in\cG$ and gap $J\in\mathcal J$, the image $g(J)$ either belongs to $\cJ$ or is contained in an open set $V_s$, for some $s\in S$ (with this last possibility occurring at least for one gap, by the previous remark).
To contour this, we consider refinements $\Theta_k$ of $\Theta_{DKN}$. Indeed, by construction of $\Theta_{DKN}$, all points of $\Delta_0$ belong to $\Lambda$ and,
denoting by $\Delta_k$ the set of endpoints of $\Theta_k$,
 as $\Delta_k=\cG(\Delta_{k-1})$ (apply \eqref{eq:refinement_Delta} inductively), one has that the union $\bigcup_{k\in\N}\Delta_k$ is the $G$-orbit of $\Delta_0$, which is dense in $\Lambda$.
 As $U_v$ is an open set, which contains points of $\Lambda$ (Theorem \ref{t:DKNmarkov3}), there exists $k\in\N$ such that $\Delta_k\cap U_v\neq\emptyset$, so that for such $k$, the $k$-th refinement $\Theta_k$ of $\Theta_{DKN}$ is proper.
\end{proof}

\begin{rem}
	One can be  more precise and show, using combinatorial arguments, that the \emph{first refinement} $\Theta_1$ is always proper.
\end{rem}

We are finally in position to give the proof of the main result.

\begin{proof}[Proof of Theorem~\ref{t:main}]
Let $\Theta_k$ be the proper $k$-th refinement of $\Theta_{DKN}$ given by Proposition \ref{p:DKN_is_proper}. 
After Proposition~\ref{p:DKN_is_proper} and Proposition \ref{p:refinement_is_ping-pong}, the partition $\Theta_k$ is the refinement of a ping-pong partition, and hence a ping-pong partition itself.
\end{proof}

\section{Examples and non-examples of ping-pong partitions}
\label{s.examples}

In this last section, we give basic examples of ping-pong partitions and point out some difficulties and subtleties related to this notion.

\subsection{Some examples} The most basic examples of groups acting with ping-pong partitions are free groups (see Section \ref{s:DKN}). We refer to \cite[Section  3]{MarkovPartitions2} and \cite{isolated} for classical and exotic examples of such actions. The examples in Figure \ref{Example1} are taken from \cite[Section  3]{MarkovPartitions2}.

\begin{figure}[ht!]
	\[
	\includegraphics{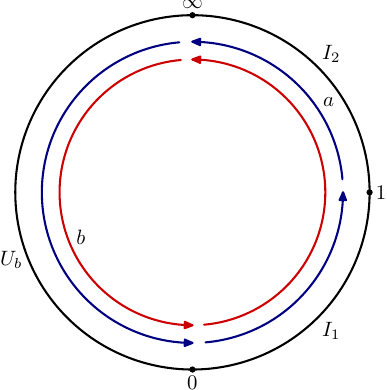}\hspace{2em}
	\includegraphics{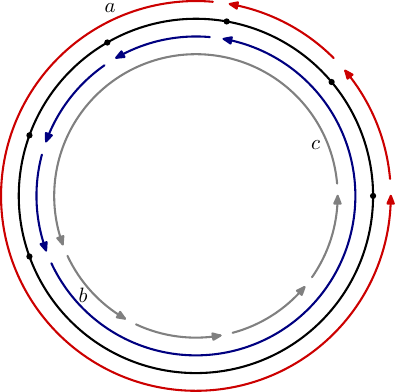}
	\]
	\caption{On the left: the natural action of $\PSL(2,\Z)\cong\Z_2*\Z_3$ associated with the Farey tesselation of the disc. On the right: an action of the free product $\Z_3*\Z_4*\Z_5$}\label{Example1}
\end{figure}

Even for the simplest free product $\PSL(2,\Z)\cong\Z_2*\Z_3$, it is possible to exhibit examples of ping-pong partitions which are \emph{exotic} (not semi-conjugated to the one in Figure \ref{Example1}, left). This has been observed by Matsumoto \cite[Section 14]{matsumoto-psl}. A not-very-enlightening picture is given in Figure \ref{matsumoto}. Here $a$ denotes the generator of the free factor $\Z_3$ and $b$ the generator of the other free factor. The ping-pong partition $\Theta=\{U_a,U_b\}$ consists of 27 disjoint intervals: the generator $a$ moves any interval of this partition cyclically across 9 intervals, while the generator $b$ moves any of the 9 connected component of $U_b$ either to a single connected component of $U_a$, or to the union of two adjacent components of $U_a$ and their common boundary point.

\begin{figure}[ht!]
	\[
	\includegraphics{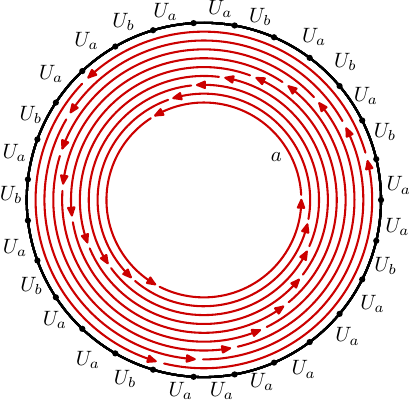}\hspace{2em}
	\includegraphics{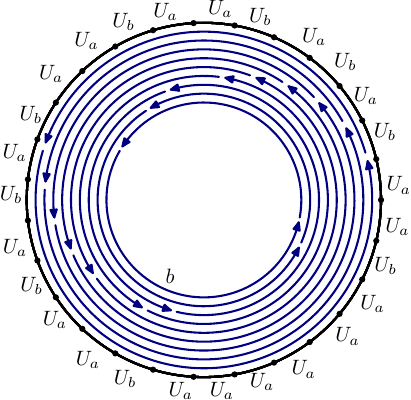}	
	\]
	\caption{The exotic ping-pong partition for $\PSL(2,\Z)\cong\Z_2*\Z_3$ described by Matsumoto.}\label{matsumoto}
\end{figure}

Of course there are examples beyond free groups and free products: after Theorems \ref{t:main} and \ref{t:RealVirtFree}, virtually free groups acting on the circle with a ping-pong partition are exactly the free-by-finite cyclic groups. Below we give examples of an amalgamated product and an HNN extension, also taken from \cite[Section  3]{MarkovPartitions2}.

\begin{figure}[ht!]
	\[
	\includegraphics{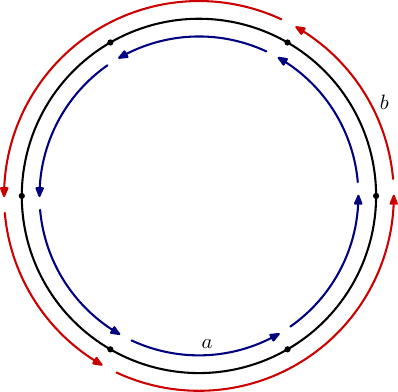}\hspace{2em}
	\includegraphics{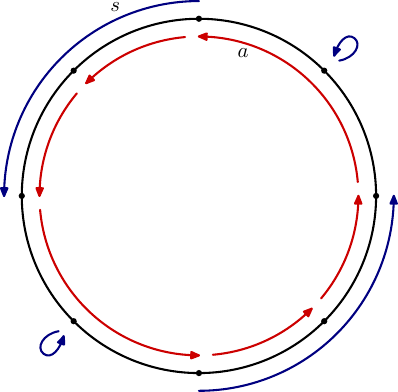}
	\]
	\caption{On the left: an action of $\SL(2,\Z)\cong\Z_4*_{\Z_2}\Z_6$. On the right: an action of the HNN extension $\Z_6*_{\Z_2}$.}\label{Example2}
\end{figure}

Both examples above are quite special because they are finite central extensions of free products. Note that there is no more general virtually free amalgamated products and HNN extensions acting on the circle with a ping-pong partition (see Remark \ref{rem:particular_cases}). A remarkable example appears in the PhD thesis of Jo\~ao Carnevale \cite[\S 7.1]{theseJoao} (see Figure \ref{fig.exotic}): the group has the presentation
\[
G=\langle r,f,g\mid r^2=\mathrm{id},frf^{-1}=r\rangle,
\]
which is the fundamental group of a graph of groups with a single vertex, with associated vertex group $\Z_2$ (generated by $r$), and two edges (corresponding to the generators $f$ and $g$).
It contains an index-two subgroup which is free of rank 3, generated by $f$, $g$, and $\overline{g}=rgr^{-1}$, so that $G\cong F_3\rtimes \Z_2$, where the action of $\Z_2$ is given by the transposition $(g,\overline g)$. The partition consists of 8 cyclically-ordered intervals $I_i=(x_i,x_{i+1})$, for $i\in \Z_8$, with the following relations:
\begin{align*}
f(I_1)&=I_1\cup \{x_2\}\cup I_2\cup \{x_3\}\cup I_3,\\
f(I_5)&=I_5\cup \{x_6\}\cup I_6\cup \{x_7\}\cup I_7,\\
f^{-1}(I_4)&=I_2\cup\{x_3\}\cup I_3\cup \{x_4\} \cup I_4,\\
f^{-1}(I_8)&=I_6\cup\{x_7\}\cup I_7\cup \{x_8\} \cup I_8,\\
g(I_7)&=I_7\cup \{x_8\}\cup I_8\cup \{x_1\}\cup I_1\cup \{x_2\}\cup I_2 \cup \{x_3\}\cup I_3\cup \{x_4\}\cup I_4\cup \{x_5\}\cup I_5,\\
g^{-1}(I_6)&=I_8\cup\{x_1\}\cup I_1\cup \{x_2\}\cup I_2\cup \{x_3\}\cup I_3\cup \{x_4\}\cup I_4 \cup \{x_5\}\cup I_5\cup \{x_6\}\cup I_6,\\
r(I_i)&=I_{i+4},\quad \text{for every }i\in \Z_8.
\end{align*}
\begin{figure}[ht!]
	\[\includegraphics[width=0.4\textwidth]{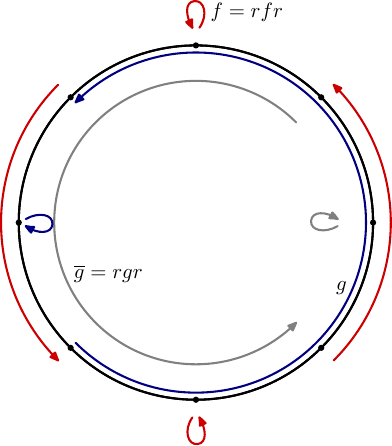}\]
	\caption{Carnevale's exotic example. The group contains a rank-3 free subgroup of index 2, and it has the property that any non-trivial element has at most 2 fixed point, although the action is not semi-conjugate to a Fuchsian action.}\label{fig.exotic}
\end{figure}
It has been proved in \cite{theseJoao} that the ping-pong partition for $G$ can be realized by a group of real-analytic circle diffeomorphisms, acting minimally, and with the property that any non-trivial element has at most two fixed points, although the action is conjugate neither to the action of a subgroup $\PSL(2,\R)$, nor to the action of a subgroup of its double cover $\SL(2,\R)$.

Illustrating examples of more complicated groups, beyond free products, faces the unavoidable difficulty given by the huge number of intervals that would be involved.

\subsection{Pointing out some difficulties}

We conclude this text with an example that explains that the properties defining a ping-pong partition are extremely delicate. Simply requiring that the circle is partitioned into finitely many intervals, with generators mapping them in an appropriate way, is not enough for us: the generators must come from a marking $(\alpha:G\to\Isom_+(X),T)$ of the group. The point is that we have to be sure that for every atom of the partition there is no ambiguity on which generators perform contractions inside it.

As an illustration, consider the group
\[G=F_2\rtimes \Z_4=\langle f,g\mid g^4=\mathrm{id},\, g^2fg^2=f^{-1}\rangle,\]
with generators acting as in Figure~\ref{fig:Basic_Example} (left): 
$g$ is the rotation $R_{1/4}$, and $f$ is a North-South map.
Although $G$ is virtually free, the previous presentation does not come from a marking. However, replacing the generator $f$ with $h=g^2f$, one obtains the presentation
\[G\cong \Z_2*\Z_4=\langle g,h\mid g^4=\mathrm{id},\,h^2=\mathrm{id}\rangle,\] which comes from a marking.
The partition in Figure~\ref{fig:Basic_Example} (left) is then replaced by that in Figure~\ref{fig:Basic_Example} (right), which is now a ping-pong partition in our sense.

\begin{figure}[ht!]
	\[
	\includegraphics{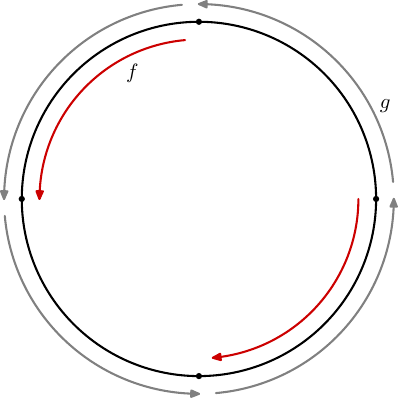}\hspace{2em}
	\includegraphics{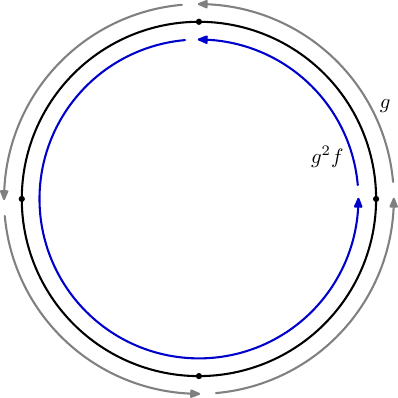}
	\]
	\caption{Two partitions for the same action of the group $G=F_2\rtimes\Z_4\cong \Z_2*\Z_4$. Only the right one is a ping-pong partition in our sense.}\label{fig:Basic_Example}
\end{figure}

\section*{Acknowledgments}

The results of this work and its companion \cite{MarkovPartitions2} were first announced in 2016. Since then, discussions with many colleagues have helped us  to find the good formalism. For this we wish to thank 
Pablo Barrientos,
Christian Bonatti,
H\'el\`ene Eynard-Bontemps,
Victor Kleptsyn,
Luis Paris,
Rafael Potrie
and Maxime Wolff.
We also warmly thank the institutions which have welcomed us during these years: the institutions of Rio de Janeiro PUC, IMPA, UFF, UFRJ, UERJ, the Universidad de la Rep\'ublica of Montevideo, the Institut de Math\'ematiques de Bourgogne of Dijon, and the Laboratoire Paul Painlev\'e of Lille.

\footnotesize{
S.A. was partially supported by IMPA, ANII via the Fondo Clemente Estable (projects FCE\_135352 and \newline FCE\_3\_2018\_1\_148740), by CSIC, via the Programa de Movilidad e Intercambios Acad\'emicos, by the MathAmSud project RGSD  (Rigidity and Geometric Structures in Dynamics) 19-MATH-04, by the LIA-IFUM and by the Distinguished Professor Fellowship of FSMP. He also acknowledges the project  ``Jeunes
Géomètres'' of F. Labourie (financed by the Louis D. Foundation).

C.M. was partially supported by CAPES postdoc program, Brazil (INCTMat \& PNPD 2015--2016), the
Ministerio de Economía y Competitividad and the Ministerio de Ciencia e Innovación, Grants MTM2014-
56950--P (2015--2017) and PID2020-114474GB-I00 (Spain UE) \& the Programa Cientista do Estado do Rio de Janeiro, FAPERJ, Brazil (2015--2018). He also received support from the  MathAmSud 2019-2020 CAPES-Brazil (“Rigidity and Geometric Structures on Dynamics”), the CNPq research grant 310915/2019-8 and Ministerio de Ciencia e Innovación PID2020-114474GB-I00 (Spain, UE) and CSIC Uruguay, via the Programa de Movilidad e Intercambios Acad\'emicos.

M.T. was partially supported by PEPS -- Jeunes Chercheur-e-s -- 2017 and 2019 (CNRS) and the R\'eseau France-Br\'esil en Math\'ematiques, MathAmSud RGSD  (Rigidity and Geometric Structures in Dynamics) 19-MATH-04. He also acknowledges INRIA Lille and the project ``Jeunes
Géomètres'' of F. Labourie (financed by the Louis D. Foundation) for the two semesters of delegation in the academic year 2018/19, during which this work has been concluded.
}

\bibliographystyle{plainurl}
\bibliography{biblio2}

\begin{flushleft}

{\scshape Juan Alonso}\\
CMAT, Facultad de Ciencias, Universidad de la Rep\'ublica\\
Igua 4225 esq. Mataojo. Montevideo, Uruguay.\\
email: \texttt{juan@cmat.edu.uy}

\smallskip	
	
{\scshape S\'ebastien Alvarez}\\
CMAT, Facultad de Ciencias, Universidad de la Rep\'ublica\\
Igua 4225 esq. Mataojo. Montevideo, Uruguay.\\
email: \texttt{salvarez@cmat.edu.uy}

\smallskip

{\scshape Dominique Malicet}\\
Laboratoire d'Analyse et de Math\'ematiques Appliquées (LAMA, UMR 8050)\\
Universit\'e Gustave Eiffel\\
5  bd.~Descartes, 77454  Champs  sur Marne, France\\
email: \texttt{dominique.malicet@crans.org}

\smallskip

{\scshape Carlos Meni\~no Cot\'on}\\
CITMAGA, 15782 Santiago de Compostela, Spain;\\
and Departamento de Matem\'atica Aplicada I\\
Universidade de Vigo, Escola de Enxe\~ner\'ia Industrial\\
R\'ua Conde de Torrecedeira 86, CP36208 Vigo, Spain\\
email:  \texttt{carlos.menino@uvigo.gal}

\smallskip

{\scshape Michele Triestino}\\
Institut de Math\'ematiques de Bourgogne (IMB, UMR 5584)\\
Universit\'e de Bourgogne Franche-Comt\'e\\
9 av.~Alain Savary, 21000 Dijon, France\\
email: \texttt{michele.triestino@u-bourgogne.fr}
\end{flushleft}

\end{document}